\documentclass[12pt]{article}
\usepackage{preprint-layout} % estetics
\usepackage{preprint-notation} % standard notation

\usepackage{hyperref}
\title{Hybridized Isogeometric Method for Elliptic Problems on CAD Surfaces with Gaps} 

\author{Tobias Jonsson \hspace{5mm} Mats G. Larson \hspace{5mm} Karl Larsson}

\date{\today}

\begin{document}

\maketitle

\begin{abstract}
We develop a method for solving elliptic partial differential equations on surfaces described by CAD patches that may have gaps/overlaps. The method is based on hybridization using a three-dimensional mesh that covers the gap/overlap between patches. Thus, the hybrid variable is defined on a three-dimensional mesh, and we need to add appropriate normal stabilization to obtain an accurate solution,  which we show can be done by adding a suitable term to the weak form. In practical applications, the hybrid mesh may be conveniently constructed using an octree to efficiently compute the necessary geometric information. We prove error estimates and present several numerical examples illustrating the application of the method to different problems, including a realistic CAD model. 
\end{abstract}

\section{Introduction}
CAD models describe surfaces using a collection of patches that meet in curves and points.  Ideally, the CAD surface is watertight, but in practice, there are often gaps or overlaps between neighboring patches. These gaps/overlaps may cause serious meshing and finite element analysis problems and in practical applications the CAD model often needs to be corrected before meshing is possible. This paper develops a robust isogeometric method \cite{IGABook} for handling CAD surfaces with gaps/overlaps.  The main idea is to 
cover the gaps/overlaps with a three-dimensional mesh and then use a hybrid variable on this mesh together with a Nitsche-type formulation.  
The hybrid variable transfers data between neighboring patches, and there is no direct communication between the patches.   To 
obtain a convergent method, the hybrid variable must be given enough stiffness in the directions normal to the interface.  We show that this 
can be done by adding a suitable term to the weak statement.  
We allow trimmed patches and add appropriate stabilization terms to control the behavior of the finite element functions in the vicinity of the 
trimmed boundaries using techniques from CutFEM,  see \cite{BurCla15}.  In practice,  we suggest an octree structure for setting up the hybrid mesh to facilitate efficient computation of the involved terms. We allow standard conforming finite element spaces as well as spline 
spaces with higher regularity.  We derive error estimates and present several numerical examples illustrating the method's convergence and application to a realistic CAD model.

\paragraph{Related Work.}
A framework that is also based on a patchwise parametrically described geometry combined with a Nitsche type method to couple the solution over patch interfaces is the discontinuous Galerkin isogeometric analysis \cite{MR3630844,MR3643563}, which considers gaps/overlaps in \cite{MR3566910, MR3547686}. One major difference to the present work is that the method involves the explicit construction of a parametric map between corresponding points over interfaces with gaps, which in our method is implicit through the stabilization of the hybrid variable. In our view the hybridized approach leads to a considerably more convenient and robust implementation that also has the benefit of supporting interfaces coupling more than two patches, cf. \cite{HanJonLar17}.
Our usage of the hybrid variable resembles the bending strip method for Kirchhoff plates \cite{MR2672111}, in which strips of fictitious material with unidirectional bending stiffness and zero membrane stiffness are placed to cover the gaps and are used for coupling the solution over the patch interfaces.
The coupling of solutions over imperfect interfaces is also addressed in overlapping mesh problems where the solution is defined on two separate meshes whose boundaries do not match, but rather intersect each other's meshes. This was extended to gaps in \cite{MR2431595,MR2344045} where elements close to the interface were modified to cover the gap, eliminating the gap regions and creating an overlapping mesh situation instead. However, it is not clear how overlapping mesh techniques could be utilized to couple solutions on surfaces since the patch meshes do not necessarily lie on the same smooth surface.

\paragraph{Outline.}

The paper is organized as follows: In Section 2 we present the method,  in Section 3 we show stability and error estimates, and in Section 4 
we present numerical experiments and examples.

\section{Model Problem and Method}

The main contribution of this paper is the robust coupling of solutions over patch interfaces with gaps/overlaps. To simplify the derivation and analysis of the method, we consider a simplified model problem that allows us to focus on the central issue and avoid complicated notation and unrelated technical arguments. We include remarks and references on how the method is extended to more general problems on CAD surfaces.

\subsection{Model Problem}
We introduce a two-dimensional model problem with a gap at an internal interface,  derive a hybridized formulation 
and the corresponding finite element method,  together with the necessary notation to proceed with the analysis.

\paragraph{Model for a Domain with Gap.} We introduce the following set-up and notation, illustrated in Figure~\ref{fig:model-domain}:
\begin{itemize}
\item Consider a domain $\Omega\subset \IR^2$ and let $\Omega_1$ and $\Omega_2$ be a partition of $\Omega$ into two subsets separated by a smooth interface $\Gamma$,  such that $\Omega_1$ is the exterior domain and $\Omega_2$ is the interior domain.  Let $U_\delta(\Gamma)\subset\IR^3$ be the open three-dimensional tubular neighborhood of $\Gamma$ with thickness $2 \delta$.  Then there is $\delta_0>0$ such that the closest point mapping $p_\Gamma:U_{\delta_0} (\Gamma) \rightarrow \Gamma$ is well defined.  %We also assume that $U_{\delta_0}(\Gamma) \subset \Omega$.

\item Let $\Omega_{i,\delta}$ be obtained by perturbing $\Gamma$ in the normal 
direction by a function $\gamma_i \in C(\Gamma)$ such that 
\begin{equation}\label{eq:gapbound}
\| \gamma_i \|_{L^\infty(\Gamma)} \lesssim \delta \leq \delta_0
\end{equation}
 More precisely 
\begin{align}
\partial \Omega_{i,\delta} = \bigcup_{x \in \Gamma}  x + \gamma_i(x)  n_\Gamma(x)
\end{align}
where $n_\Gamma(x)$ is the unit normal to $\Gamma$ exterior to $\Omega_2$.  Note that the 
functions $\gamma_1$ and $\gamma_2$ are different and therefore the domains 
$\Omega_{1,\delta}$ and $\Omega_{2,\delta}$ do not perfectly match at the interface,  instead there 
may be a gap or an overlap but in view of (\ref{eq:gapbound}) we will have 
\begin{align}
\partial \Omega_{1,\delta} \cup \partial \Omega_{2,\delta} \subset U_{\delta}(\Gamma) 
\subset U_{\delta_0}(\Gamma)
\end{align}
\end{itemize}

\begin{figure}
\centering
\begin{subfigure}[t]{0.27\linewidth}\centering
\includegraphics[width=0.9\linewidth]{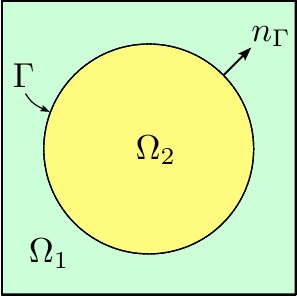}
\subcaption{Two patch domain}
\label{fig:model-a}
\end{subfigure}
\begin{subfigure}[t]{0.27\linewidth}\centering
\includegraphics[width=0.9\linewidth]{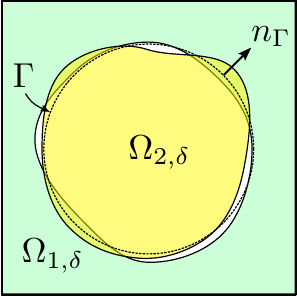}
\subcaption{Perturbed patches}
\label{fig:model-b}
\end{subfigure}
\begin{subfigure}[t]{0.27\linewidth}\centering
\includegraphics[width=0.9\linewidth]{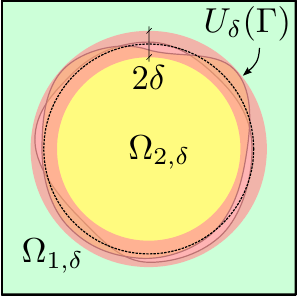}
\subcaption{Tubular neighborhood}
\label{fig:model-c}
\end{subfigure}
\\[1.0ex]	
%\begin{subfigure}[t]{0.47\linewidth}\centering
%\includegraphics[width=0.9\linewidth]{model-problem-3d}
%\subcaption{Two-dimensional model problem}
%\label{fig:model-ax}
%\end{subfigure}
\begin{subfigure}[t]{0.7\linewidth}\centering
\includegraphics[width=0.8\linewidth]{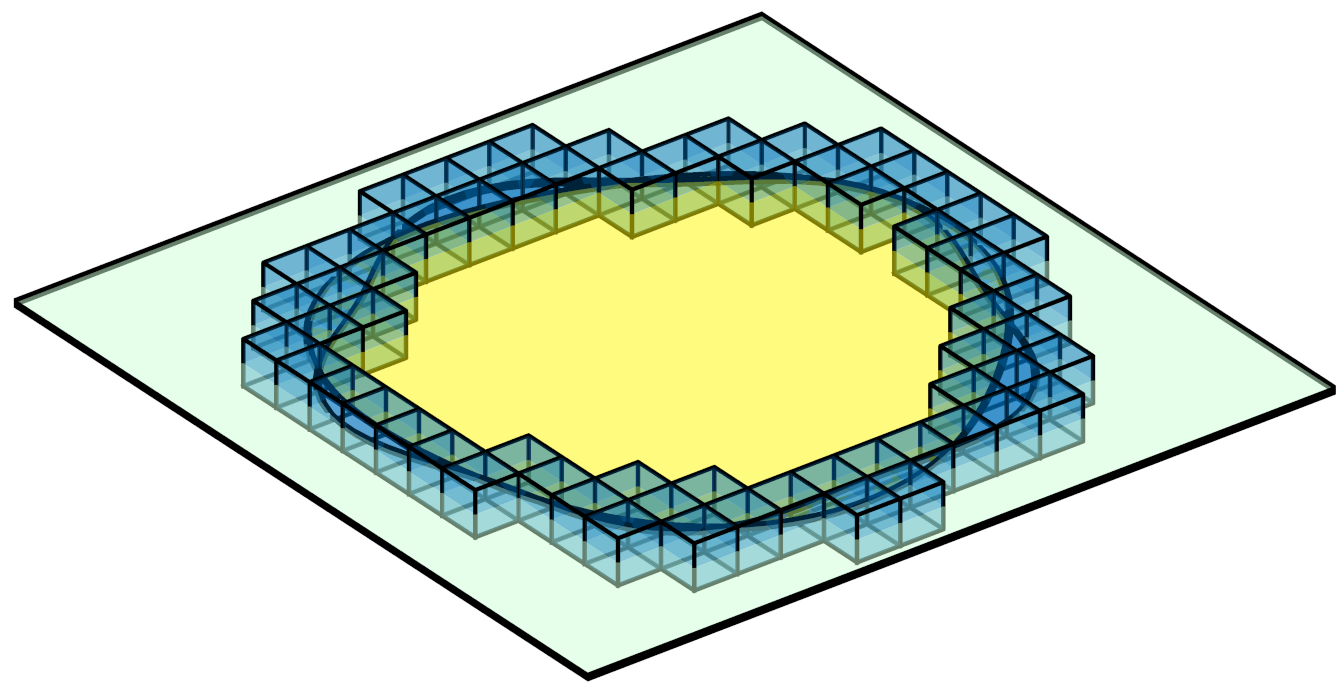}
\subcaption{Three-dimensional hybrid mesh covering the interface}
\label{fig:model-bx}
\end{subfigure}
\caption{\emph{Model problem with gap/overlap.} \emph{Top:} In the derivation and analysis of the method we use this conceptual construction of a two-dimensional two-patch domain with gaps/overlaps stemming from perturbation of the patch boundaries facing the interface.
\emph{Bottom:} While the perturbed two-patch domain entirely lives in the two-dimensional plane, the hybrid variable for increased generality will live on a three-dimensional mesh covering the imperfect interface.}
\label{fig:model-domain}
\end{figure}

\paragraph{Exact Model Problem.}
Consider the following model interface problem on the exact partition of $\Omega$ (without a gap/overlap): Find $u$ fulfilling
\begin{align} \label{eq:model-problem}
-\Delta u_i = f_i \qquad \text{in $\Omega_i$},\quad i=1,2
\end{align}
with interface conditions
\begin{align} \label{eq:model-problem-interface}
u_1 = u_2, \qquad \nabla_{n_1} u_1 + \nabla_{n_2} u_2 = 0 \qquad \text{on $\Gamma$} 
\end{align}
and a homogeneous Dirichlet boundary condition $u=0$ on $\partial\Omega$.
Here $u_i$ indicates the solution on the patch $\Omega_i$, and we let $u_0$ denote the solution on the interface $\Gamma$. We assume a regularity of the weak solution on each patch $u_i \in H^{s}(\Omega_i) \cap H^1_0(\Omega)|_{\Omega_i}$, where $s>3/2$. Further, for the solution on the interface we assume $u_0 \in H^{s}(\Gamma)$, which is likely $1/2$ more regularity than is required since this is essentially the trace along $\Gamma$ but we maintain this assumption for simplicity. In summary, we assume a weak solution with the following decomposition into three fields
\begin{align}
u = (u_0; u_1; u_2) \in W &= V_0 \otimes V_1 \otimes V_2
%\\&
= \bigl( H^s(\Gamma) \otimes H^s(\Omega_1) \otimes H^s(\Omega_2) \bigr) \cap H^1_0(\Omega)
\end{align}

\paragraph{Extended Solution.}
We will next derive a weak formulation on the perturbed patches $\Omega_{i,\delta}$ instead of on the exact patches $\Omega_i$. To make sense of the exact solution $u$ in such a formulation we must first extend $u$ to the perturbed domains. We recall that there is an extension operator $E_i : H^s(\Omega_{i}) \rightarrow H^s(\IR^2)$, independent of $s$,
such that
\begin{align}
\| E_i v \|_{H^s(\IR^2)} \lesssim \| v \|_{H^s(\Omega_{i})} 
\end{align} 
and $E_i v = v$ on $\Omega_i$, see \cite{stein70}. For the derivation of the hybridized formulation we introduce fields $u_0,v_0$ defined on a domain $\Omega_0 \subset \IR^3$ fulfilling 
\begin{equation} \label{eq:extension-patch}
\partial \Omega_{1,\delta} \cup \partial \Omega_{2,\delta} \cup \Gamma 
\subset \Omega_0 \subset U_{\delta_0}(\Gamma)
\end{equation}
and hence we must also extend the exact solution $u$ on $\Gamma$ to $\Omega_0$. To this end we define an extension $E_0: H^s(\Gamma) \rightarrow H^s(U_{\delta_0}(\Gamma))$ 
such that $(E_0 v)|_x = v\circ p_\Gamma (x)$. Clearly, $E_0 v = v$ on $\Gamma$.  We then have
\begin{align} \label{eq:extension-interface}
\| E_0 v \|_{H^s(U_{\delta_0}(\Gamma))} \lesssim \delta_0 \| v \|_{H^s(\Gamma)}
\end{align}
see \cite{BurHanLarMas18}.
For compactness we introduce the notation
\begin{align}
u^e = (u_0^e; u_1^e; u_2^e) = (E_0 u_0; E_1 u_1; E_2 u_2)
\end{align}
where it is implied by the subscript of the field which extension operator is used.
We also apply this notation to spaces such that, for instance, $W^e = \{ v = w^e \,:\, w \in W \}$.

\paragraph{Hybridized Weak Formulation.}
Since an extended function coincides with the original function on its original domain, we may replace the fields in the continuous problem \eqref{eq:model-problem}--\eqref{eq:model-problem-interface} by their extensions.
We then, patchwise, multiply \eqref{eq:model-problem} by a test function $v_i^e \in V_i^e$, integrate over the perturbed patch $\Omega_{i,\delta}$, and apply a Green's formula to obtain
\begin{align}
\sum_{i=1}^2 (f_i^e,v_i^e)_{\Omega_{i,\delta}} &= \sum_{i=1}^2 (-\Delta u_i^e, v_i^e)_{\Omega_{i,\delta}} 
\\
&= \sum_{i=1}^2 (\nabla u_i^e, \nabla v_i^e)_{\Omega_{i,\delta}} 
- (\nabla_n u^e, v_i^e)_{\partial \Omega_{i,\delta}} 
\\
&= \sum_{i=1}^2 (\nabla u_i^e, \nabla v_i^e)_{\Omega_{i,\delta}} - (\nabla_n u_i^e, v_i^e - v_0^e)_{\partial \Omega_{i,\delta}} -  (\nabla_n u_i^e,  v_0^e)_{\partial \Omega_{i,\delta}}
\\
&\approx \sum_{i=1}^2 (\nabla u_i^e, \nabla v_i^e)_{\Omega_{i,\delta}} - (\nabla_n u_i^e, v_i^e - v_0^e)_{\partial \Omega_{i,\delta}}
- (u_i^e - u_0^e, \nabla_n v_i^e)_{\partial \Omega_{i,\delta}} 
\\
&\qquad\quad 
+ \beta h^{-1} ( u_i^e - u_0^e, v_i^e - v_0^e)_{\partial \Omega_{\delta,i}}
- (\nabla_n u_i^e,  v_0^e)_{\partial \Omega_{i,\delta}}
\end{align}
where we added and subtracted functions $u_0^e = u_0 \circ p_\Gamma = u|_\Gamma \circ p_\Gamma$ and $v_0^e$, and in the 
last step we added terms involving $u^e - u_0^e$ that are not exactly zero since they are 
evaluated on the perturbed curves $\partial \Omega_{i,\delta}$, which differ from $\Gamma$.  
The functions $u_0^e$ and $v_0^e$ will, due to the construction of $E_0$ using the closest point mapping $p_\Gamma$, in the continuous problem be constant in the directions orthogonal to $\Gamma$. In the discrete setting, this property will instead be imposed weakly since it is not straightforward to implement strongly.

\paragraph{Application to Surfaces.} The model problem can be directly extended to a setting with a surface built up by a set of patches,  
$\mcO = \{ \Omega_i : i \in I\}$ with $I$ an index set,  and interfaces $\{\Gamma_{ij} = \partial \Omega_i \cap \partial \Omega_j\}$.  The patches 
are defined by a mapping $F_i : \IR^2 \supset \widehat{\Omega}_i \rightarrow \Omega_i \subset \IR^3$,  and a set of trim curves 
$\widehat{\Gamma}_{ij}$.  In the model problem \eqref{eq:model-problem} the Laplace operator is replaced by the  Laplace-Beltrami operator $\Delta_\Omega$,  the gradients are replaced by tangential gradients $\nabla_\Omega$,  and 
the interface conditions are
\begin{align} \label{eq:surf-interface-cond}
u_i = u_j, \qquad \nabla_{\nu_i} u_i + \nabla_{\nu_j} u_j = 0 \qquad \text{on $\Gamma_{ij}$} 
\end{align}
where  and $ \nabla_{\nu_i} = \nu_i \cdot \nabla_\Omega$ are the tangential derivatives along the exterior unit co-normals $\nu_i$ to $\partial \Omega_{i,\delta}$. Note that here $\nu_i$ may be different from $-\nu_j$ and 
thus $\Gamma_{ij}$ may be a sharp edge on the surface across which the surface normal is discontinuous.  The perturbation of the surface may 
be precisely defined by first extending $\Omega_i$ to a slightly larger smooth surface $\widetilde{\Omega}_i$ and then assuming that 
$\partial \Omega_{i,\delta}$ is smooth curve on $\widetilde{\Omega}_i$ such that 
\begin{align}\label{eq:pert-int}
\partial \Omega_{i,\delta} \subset U_\delta(\Gamma)
\end{align}
The surface patches can be further perturbed by the action of a rigid body motion in $\IR^3$ with norm less than $\delta$.   The analysis we 
present is basically directly applicable to this setting since the key assumption is  (\ref{eq:pert-int}). A further difficulty that we do not consider 
here is a more general perturbation of the mapping $F$.  We have chosen to present the method 
and analysis in the simple setting outlined in the previous paragraph since it captures the main new challenges and the notation is much simpler.

\paragraph{Implementation.} In practice we first import a number of patches that do not match perfectly.  These patches $\{ \Omega_i : i \in I \}$ 
are each described by the mapping $F_i$ together with a set of trim curves $\{ \gamma_j : j \in J_I \}$ defining the boundary of the patch in 
the reference domains.  We then compute the intersection with the mapped trim curves $F(\gamma_i)$ and voxels in an octree which allows local refinement.  We can then extract a suitable cover of the gaps between the mapped patches consisting of a face-connected set of voxels which is the mesh used for the hybrid variable.  The precise formulation of such algorithms is not the focus of this paper and we leave that for future work. Note, in particular, that no information is passed directly between two patches instead all information is passed through the hybrid variable.
%Finally,  
%we refer to \cite{JonLarLar17,MR4021278} for further details on finite element formulations on multipatch CAD surfaces.\todo{fix last sentence}
 
\subsection{Hybridized Finite Element Method}
\paragraph{Finite Element Spaces.}
To define the finite element spaces we assume that we have polygonal domains $\Omega_i \subset \widetilde{\Omega}_{i} \subset \IR^2$ 
and families of quasiuniform meshes  $\widetilde{\mcT}_{h,i} $ on $\widetilde{\Omega}_i$ with mesh parameter $h \in (0,h_0]$,  for $i=1,2.$ 
We define the active meshes and the corresponding discrete domains by
\begin{equation}
\mcT_{h,i} = \{ T \in \widetilde{\mcT}_{h,i} : T \cap \Omega_i \neq \emptyset \},  \qquad \Omega_{h,i} = \cup_{T \in \mcT_{h,i} }T ,  \qquad i =1,2
\end{equation}  
For the hybrid mesh we instead consider a polygonal domain $  \Gamma \subset U_{\delta_0}(\Gamma) \subset \widetilde{\Omega}_0 \subset \IR^3$ 
and a family of quasiuniform meshes $\widetilde{\mcT}_{h,0}$ on $\widetilde{\Omega}_{h,0}$ with mesh parameter $h \in (0,h_0]$. Then we define the active mesh by 
\begin{equation}
\mcT_{h,0} = \{ T \in \widetilde{\mcT}_{h,0} : T \cap U_\delta(\Gamma) \neq \emptyset \}, \qquad \Omega_{h,0} = \cup_{T \in \mcT_{h,0}} T 
\end{equation}  
Next we let $\widetilde{V}_{h,i}$ be a conforming finite element or spline space on $\widetilde{\mcT}_{h,i}$ and we define the active finite 
element spaces by restriction to the active mesh
\begin{equation}
V_h = \widetilde{V}_h |_{\Omega_{h,i},} \qquad i = 0,1,2
\end{equation}
Finally,  the finite element space is the direct sum of our three spaces
\begin{align}
W_h = V_{h,0} \oplus V_{h,1} \oplus V_{h,2}
\end{align}
Here we emphasize that the space $V_{h,0}$ is defined on the three-dimensional mesh $\mcT_{h,0}$ and the spaces $V_{h,i}$ are defined on the two dimensional meshes $\mcT_{h,i}$,  $i=1,2.$

\paragraph{Definition of the Method.}
Based on the derivation we define the method: find $u_{h} \in W_h$ such that  
\begin{equation} \label{eq:method}
A_h(u_h,v) = l_h(v)\qquad \forall v \in W_h
\end{equation}
where 
\begin{align}
A_h(v,w) &= s_{h,0}(v_0,w_0) + \sum_{i=1}^2 a_{h,i}(v_0,v_i;w_0,w_i) + s_{h,i}(v_i,w_i)
\\
l_h(v) &=  \sum_{i=1}^2 (f_i^e, v_i)_{\Omega_{i,\delta}}
\\
\shortintertext{and we have the hybrid variable stabilization}
\label{eq:blockstab}
s_{h,0}(v_0,w_0) &= \tau_0
h^{-\alpha}
\Bigl(
( \nabla_\Gamma^\perp v_0,\nabla_\Gamma^\perp w_0)_{\mcT_{h,0}} 
+
\sum_{l=1}^p h^{2l + 1} \bigl( \llbracket \nabla_n^l v_0 \rrbracket , \llbracket \nabla_n^l w_0 \rrbracket \bigr)_{\mcF_{h,0}}
\Big)
\end{align}
where $\alpha$ is a parameter, $\nabla_\Gamma^\perp$ is the component  of $\nabla_{\IR^3}$ normal to $\Gamma$, and $\llbracket \nabla_n^l w \rrbracket$ denotes the jump over a face in the $l$:th directional derivative of $w$ in the direction of the face normal. The forthcoming analysis shows that $\alpha = 2$ is a suitable choice. The remaining forms are defined by
\begin{align} \label{eq:ahi}
a_{h,i}(v_0, v_i; w_0, w_i) &= (\nabla v_i , \nabla w_i)_{\Omega_{i,\delta}} 
%\\ \nonumber&\qquad 
- (\nabla_n v_i, w_i - w_0)_{\partial \Omega_{i,\delta}} 
- (v_i - v_0, \nabla_n w_i)_{\partial \Omega_{i,\delta}}
\\ \nonumber
&\qquad 
+ \beta h^{-1} ( v_i - v_0, w_i - w_0)_{\partial \Omega_{\delta,i}}
\\ \label{eq:ghost-penalty}
s_{h,i}(v_i,w_i) &= \tau_i \sum_{l=1}^p h^{2l-1} \bigl( \llbracket \nabla_n^l v_i \rrbracket , \llbracket \nabla_n^l w_i \rrbracket \bigr)_{\mcF_{h,i}}
\end{align}
where $\tau_0,\tau_1,\tau_2$ and $\beta$ are positive parameters.
For simplicity,  we do not consider the implementation of the Dirichlet boundary condition on the 
exterior boundary $\partial \Omega$. We could either assume that we have a matching mesh at $\partial \Omega$ and use strong boundary conditions 
or use a weak Nitsche-type method.

\begin{rem}[Hybrid Variable Stabilization] \label{rem:hybrid-stab}
The first term in the stabilization \eqref{eq:blockstab} of the hybrid variable is the most important and provides the necessary control of the 
variation of the hybrid variable across the gap,  see estimate \eqref{eq:v0-pert} below.  The second term is added to increase 
robustness and the well-conditioning of the algebraic equations.
In the first term we must be able to evaluate the gradient $\nabla_\Gamma^\perp = (I - t_\Gamma \otimes t_\Gamma)\nabla_{\IR^3}$, where $t_\Gamma$ is the tangent to $\Gamma$, extended to the complete hybrid variable domain $\Omega_{h,0}$. One option is to extend $t_\Gamma$ to $\mcT_{h,0}$ using the closest point mapping $p_\Gamma(x)$. While $\Gamma$ in the description above is the location of the exact interface, this in most practical situations is unknown. However, since $\Gamma$ is just a theoretical construction we instead define the position of $\Gamma$ based on the perturbed interfaces, for instance as the midpoint between the closest point on $\partial\Omega_{1,\delta}$ respectively on $\partial\Omega_{2,\delta}$.
A more elaborate option would be to introduce a discrete field variable for $t_\Gamma$ on $\Omega_{h,0}$ that is determined via projection of the tangent vectors of $\partial\Omega_{i,\delta}$. Such an approach would have the benefits of not relying on identifying closest points and facilitating higher-order approximations of how information flows over the gap.
For suitable stabilization when there is no gap/overlap, see \cite{MR4021278}, where a similar patch coupling with a hybridized approach is considered.
\end{rem} 

\begin{rem}[Patch Stabilization] \label{rem:sh_i}
On each patch we include \eqref{eq:ghost-penalty}, which is a so-called ghost penalty stabilization term \cite{Bur10}. The inclusion of this stabilization allows us to use cut finite element methods \cite{BurCla15,JonLarLar17} for discretizing the solution on each patch. Essentially, the mesh on each patch is not required to conform to the patch geometry --- it is sufficient that the mesh covers the geometry ---
and still, the method enjoys the same approximation and stability properties as a standard FEM.
Alternative stabilization approaches include finite cell stabilization \cite{MR4394710} and discrete extension \cite{MR4502991}.
In a cut setting, it is natural to use a weak Nitsche-type method for implementing the Dirichlet boundary condition.
\end{rem}

\begin{rem}[Extension to Isogeometry] \label{rem:extension2isogeometry}
In the surface CAD description, each surface patch $\Omega_{i,\delta} \subset \IR^3$, is described using a parametric map $F_i:\widehat \Omega_{i,\delta} \to \Omega_{i,\delta}$ from a two-dimensional reference domain $\widehat \Omega_{i,\delta} \subset [0,1]^2$. Following the procedure outlined above for extension to surfaces, we then patchwise transform the problem back to $\widehat \Omega_{i,\delta}$ before discretizing. For instance, this means that the form corresponding to \eqref{eq:ahi} will take the structure
\begin{align}
\label{eq:ahi-surf}
a_{h,i}(v_0,v_i;w_0, w_i) &= (|G_i|^{1/2} G_i^{-1}\nabla v_i , \nabla w_i)_{\widehat\Omega_{i,\delta}} 
\\ \nonumber &\qquad 
- (n \cdot (|G_i|^{1/2} G_i^{-1 }\nabla v_i, w_i - w_0\circ F_i)_{\widehat\partial \Omega_{i,\delta}}
\\ \nonumber &\qquad 
- (v_i - v_0\circ F_i, n\cdot(|G_i|^{1/2} G_i^{-1 }\nabla w_i))_{\widehat\partial \Omega_{i,\delta}}
\\ \nonumber &\qquad 
+ \beta h^{-1} ( |G_i|^{1/2} n\cdot G_i^{-1 } \cdot n (v_i - v_0\circ F_i), w_i - w_0\circ F_i)_{\widehat\partial \Omega_{\delta,i}}
\end{align}
where $G_i$ is the metric tensor implied by the map $F_i$. Note that the patch mesh in this case is directly defined on the two-dimensional reference domain, and so is the patch stabilization. For more details on this topic, we refer to our work in \cite{JonLarLar17}.
\end{rem}

\section{Error Estimates}

In this section, we derive an error estimate for the method applied to the model problem. To keep the complexity of the paper at a minimal level we consider the most fundamental stability and energy estimates in a situation with planar patches and a three-dimensional hybrid variable. This model problem simplifies the notation significantly and captures the essential difficulties in the analysis. The extension to curved patches that meet in a sharp edge is direct using the techniques developed in \cite{JonLarLar17} and 
\cite{HanJonLar17}. We discuss the details of these extensions in Remark \ref{rem:iga-ext} at the end of this section.

\paragraph{Norms,  Stabilization,  and Poincar\'e Inequality.}
 Define the energy norm
\begin{equation} \label{eq:energy-norm}
\tn v \tn_h^2 = \| v_0 \|^2_{s_{h,0}} + \sum_{i=1}^2 \| \nabla v_i \|^2_{\Omega_{i,\delta}} 
+ \| v_i \|^2_{s_{h,i}}
+ h \| \nabla v_i \|^2_{\partial \Omega_{i,\delta}} 
+ h^{-1} \| v_i - v_0\|^2_{\partial \Omega_{i,\delta}} 
\end{equation}
where $\| w \|^2_{s_{h,i}} = s_{h,i}(w,w)$, $i=0,1,2$, and $\| w \|_\omega^2 = \int_\omega w^2$ is the usual $L^2(\omega)$ norm.

The stabilization forms provide the control 
\begin{align}
\label{eq:sh-inverseL2}
\| \nabla^m v_i \|^2_{\Omega_{h,i}} &\lesssim \| \nabla^m v_i \|^2_{\Omega_{i,\delta}} + \| v_i \|^2_{s_{h,i}}, \qquad i=1,2, \quad m=0,1
\\
\label{eq:sh0-inverseL2}
h^{-2} \| v_0 \|^2_{\Omega_{h,0}} 
&\lesssim 
\| v_0 \|^2_{\partial\Omega_{i,\delta}} + \| v_0 \|^2_{s_{h,0}} 
\end{align}
see \cite{BurHanLarMas18, HanLarLar17, LarZah20} for proofs. We also have the following result that quantifies the control provided by the stabilization of the hybrid variable.
\begin{lem}[Hybrid Variable Control] \label{lem:hybrid-variable-control}
For $v_0 \in V_{h,0}$ and $i=1,2,$ there are bounds
\begin{alignat}{2}
\label{eq:v0-pert}
\| v_0 - v_0^e \|_{\partial\Omega_{i,\delta}}^2
&\lesssim \delta^2 h^{\alpha-2} \| v_0 \|_{s_{h,0}}^2
\\
\label{eq:v0-control}
\| v_0 \|_{\partial\Omega_{i,\delta}}^2
&\lesssim
\delta^2 h^{\alpha-2} \| v_0 \|_{s_{h,0}}^2 + \tn v \tn_h^2
\end{alignat}
where $v_0^e(x) = v_{0}\circ p_\Gamma(x)$. Assuming $\alpha\geq 2$, these bounds may be simplified since then
\begin{align} h^{\alpha-2} \| v_0 \|_{s_{h,0}}^2 \leq \tn v \tn_h^2
\end{align}
\end{lem}
\begin{proof} {\bf (\ref{eq:v0-pert}).}  Let $I(x,p_\Gamma(x))$ be the line segment connecting $x$ and $p_\Gamma(x)$.  We then have 
\begin{align}
v_0(x)  - v_{0}^e(x)
=
v_0(x)  - v_{0} (p_\Gamma(x)) = \int_{ I(x,p_\Gamma(x)) } t \cdot \nabla_\Gamma^\perp v_0
\end{align}
where $t$ is the unit tangent vector to  $I(x,p_\Gamma(x))$.  Estimating the right-hand side using a Hölder inequality 
we get 
\begin{align}
|v_0(x)  - v_{0}^e (x) | \leq \delta  \| \nabla_\Gamma^\perp v_0 \|_{L^\infty( I(x,p_\Gamma(x)) ) }
\end{align}
Squaring and integrating over $\partial\Omega_{i,\delta}$ give
\begin{align}
\| v_0 - v_{0}^e \|_{\partial\Omega_{i,\delta}}^2
&
\leq
\delta^2 \int_{\partial\Omega_{i,\delta}} \| \nabla^\perp_\Gamma v_0 \|_{L^\infty( I(x,p_\Gamma(x)))}^2 \, dx
%\\&
%\lesssim
%\delta^2 \int_{\partial\Omega_{i,\delta}} \| \nabla^\perp_\Gamma v_0 \|_{L^\infty( \tilde I(x,p_\Gamma(x))\cap \mcT_{h,0})}^2 \, dx
%\\&
%\lesssim
%\delta^2 \int_{\partial\Omega_{i,\delta}} h^{-1} \| \nabla^\perp_\Gamma v_0 \|_{L^2(\tilde I(x,p_\Gamma(x))\cap \Omega_{h,0})}^2
\\& \label{eq:here-tech-bound}
\lesssim \delta^2 h^{-2} \| \nabla^\perp_\Gamma v_0 \|_{\mcT_{h,0}}^2
\\&
\lesssim \delta^2  h^{-2} h^\alpha \| v_0 \|_{s_{h,0}}^2
\end{align}
which is our desired estimate.
In \eqref{eq:here-tech-bound} we used the following technical bound
\begin{equation}\label{eq:technical-a}
\int_{\partial\Omega_{i,\delta}} \|  w  \|_{L^\infty( \tilde I(x,p_\Gamma(x))\cap T)}^2 \, dx 
\lesssim h^{-2} \| w  \|_T^2
\end{equation}
for an element $T \in \mcT_{h,0}$, where $w \in \mathbb{P}_k(T)$,  the polynomials of degree $k$ on $T$, and $\tilde I(x,p_\Gamma(x))$ is the straight line covering $I(x,p_\Gamma(x))$.  To verify (\ref{eq:technical-a}) we first 
recall that since the elements are shape regular and the mesh quasi-uniform there are  balls $B_{r_1} \subset T \subset B_{r_2}$,  with the same 
center and radii that satisfy $r_1 \sim r_2 \sim h$.   For any line $l$ in $\IR^3$ that intersects $T$ we have the inverse inequality 
\begin{align}
\| w \|^2_{L^\infty(l \cap T)}  &\lesssim  \| w \|^2_{L^\infty(l \cap B_{r_2})}  \lesssim  \| w \|^2_{L^\infty(l \cap B_{2 r_2})} 
\\ \label{eq:tech-b}
&\qquad \lesssim  h^{-1}  \| w \|^2_{l \cap B_{2 r_2}} \lesssim  h^{-3}  \| w \|^2_{B_{2 r_2}}  \lesssim  h^{-3}  \| w \|^2_{B_{r_1}}  \lesssim  h^{-3}  \| w \|^2_T
\end{align}  
where we used the fact that the length $|l \cap B_{2 r_2}|$ of the line segment $l \cap B_{2 r_2}$ satisfy $|l \cap B_{2 r_2}| > r_2 \gtrsim h$,  an 
inverse inequality to pass from the line to the ball $B_{2 r_2}$,  and finally an inverse inequality to pass to $B_{r_1}$ which is contained in $T$ 
by shape regularity.  Using (\ref{eq:tech-b}) we get 
\begin{align}
\int_{\partial \Omega_{\delta,i}}  \| w  \|^2_{L^\infty( \tilde I(x,p_\Gamma(x))\cap T)} dx  
&\lesssim
\int_{\partial \Omega_{\delta,i}}  h^{-3} \| w  \|^2_T  dx 
\\
&\lesssim  |\partial \Omega_{\delta,i} \cap p_\Gamma^{-1} (T)|  h^{-3} \| w  \|^2_T
\\
&\lesssim   h^{-2} \| w  \|^2_T 
\end{align}
where we finally used the fact that  $|\partial \Omega_{\delta,i} \cap p_\Gamma^{-1} (T)|\lesssim h$. This completes the verification of \eqref{eq:technical-a}, and hence, the proof of \eqref{eq:v0-pert}.

\noindent{\bf (\ref{eq:v0-control}).} For $i=1$ we add and subtract $v_1 \in V_{h,1}$ and estimate using standard inequalities
\begin{align}
\| v_0 \|_{\partial\Omega_{1,\delta}}^2
&\lesssim \| v_0 - v_1 \|_{\partial\Omega_{1,\delta}}^2 + \| v_1 \|_{\partial\Omega_{1,\delta}}^2
\\
&\lesssim \| v_0 - v_1 \|_{\partial\Omega_{1,\delta}}^2 + \| v_1 \|_{\Omega_{1,\delta}}^2 + \| \nabla v_1 \|_{\Omega_{1,\delta}}^2
\\
&\lesssim \| v_0 - v_1 \|_{\partial\Omega_{1,\delta}}^2 + \| \nabla v_1 \|_{\Omega_{1,\delta}}^2
\\  \label{eq:tech-c}
&\leq \tn v \tn_h^2
\end{align}
where we used a trace inequality followed by the control provided by the  Dirichlet condition on $\partial \Omega$.  In the case $i=2$ 
we instead add and subtract $v_{0}^e$, 
\begin{align}
\| v_0 \|_{\partial\Omega_{2,\delta}}
&\leq
\| v_0 - v_{0}^e \|_{\partial\Omega_{2,\delta}}
+
\| v_{0}^e \|_{\partial\Omega_{2,\delta}}
\\&\lesssim
\delta h^{\alpha/2-1} \| v_0 \|_{s_{h,0}}
+
\| v_{0}^e \|_{\partial\Omega_{1,\delta}}
\\&\leq
\delta h^{\alpha/2-1} \| v_0 \|_{s_{h,0}}
+
\| v_0 - v_{0}^e \|_{\partial\Omega_{1,\delta}}
+
\| v_0 \|_{\partial\Omega_{1,\delta}}
\\&\lesssim \label{eq:last-stab-bound}
\delta h^{\alpha/2-1} \| v_0 \|_{s_{h,0}}
+ \tn v \tn_h
\end{align}
where we used (\ref{eq:v0-pert}),  the fact that $v_{0}$ is constant orthogonally to $\Gamma$ to pass from 
$\partial \Omega_{\delta,2}$  to $\partial \Omega_{\delta,1}$,  and then the bound (\ref{eq:tech-c}) for $i=1$. This concludes the proof of \eqref{eq:v0-control}.
\end{proof}

\begin{lem}[Poincar\'e Inequality] 
Assuming $\alpha \geq 2$, it holds % there is a constant such that
\begin{align}\label{eq:poincare}
\boxed{
h^{-2} \| v_0 \|^2_{\mcT_{h,0}} +  \sum_{i=1}^N \| v_i \|^2_{\mcT_{h,i}}  \lesssim \tn v \tn^2_h , \qquad v \in W_h
}
\end{align}
and as a consequence  $\tn \cdot \tn_h$ is a norm on $W_h$.
\end{lem}

\begin{proof}
Let  $\phi$ be the solution to the dual problem 
\begin{equation}\label{eq:poincare-dual}
\text{$-\Delta \phi = \psi$ in $\Omega$} ,
\qquad 
\text{$\phi = 0$ on $\partial \Omega$}
\end{equation} 
with $\psi \in L^2(\Omega)$, which satisfies the standard regularity estimate
\begin{equation}\label{eq:poincare-dual-reg}
\| \phi \|_{H^{2}(\Omega)} \lesssim \|\psi \|_\Omega
\end{equation}
Consider first the  estimation of the bulk 
subdomain contributions. Using \eqref{eq:sh-inverseL2} we have
\begin{align}
\label{eq:init-calc-a}
\sum_{i=1}^2 \| v_i \|^2_{\mcT_{h,i}} &\lesssim \sum_{i=1}^2 \| v_i \|^2_{\Omega_{i,\delta}} + \| v_i \|^2_{s_{h,i}}
\end{align}
where the last term is trivially bounded by $\tn v\tn_h^2$.
To estimate $\sum_{i=1}^2 \| v_i \|^2_{\Omega_{i,\delta}}$ we multiply the dual problem
\eqref{eq:poincare-dual} by $v_i \in V_{h,i}$ and then using integration by parts on 
each of the patch domains 
$\Omega_{i,\delta}$,  $i=1,2$,  we obtain 
\begin{align}\label{eq:poincare-proof-a}
 \sum_{i=1}^2 (v_i,\psi)_{\Omega_{i,\delta}} 
 &=
 \sum_{i=1}^2 (\nabla v_i, \nabla \phi )_{\Omega_{i,\delta}} - (v_i, \nabla_n \phi )_{\partial \Omega_{i,\delta}}
 \\ \label{eq:poincare-proof-b}
 &=
 \sum_{i=1}^2 ( \nabla v_i, \nabla \phi )_{\Omega_{i,\delta}} 
 - (v_i - v_0, \nabla_n \phi )_{\partial \Omega_{i,\delta}}
- (v_{0}, \nabla_n \phi )_{\partial \Omega_{i,\delta}}
 \\ \label{eq:poincare-proof-c}
 &\lesssim
  \sum_{i=1}^2 \| \nabla v_i\|_{\Omega_{i,\delta}} \|\nabla \phi \|_{\Omega_{i,\delta}}
  \\\nonumber &\qquad\quad
 + \left(\|v_i - v_{0} \|_{\partial \Omega_{i,\delta}} + \| v_{0} \|_{\partial \Omega_{i,\delta}}\right) \| \nabla \phi \|_{\partial \Omega_{i,\delta}}
  \\ \label{eq:poincare-proof-e}
 &\lesssim
(1+\delta h^{\alpha/2-1}) \tn v \tn_h \underbrace{\Big( \sum_{i=1}^2 \|\nabla \phi \|^2_{\Omega_{i,\delta}} 
 + \| \phi \|^2_{H^{2}(\Omega_{i,\delta})} \Big)^{1/2}}_{ \lesssim \| \phi \|_{H^{2}(\Omega)}}
\\ \label{eq:poincare-proof-f}
&\lesssim (1+\delta h^{\alpha/2-1}) \tn v \tn_h \| \psi \|_{\Omega}
\end{align}
where in (\ref{eq:poincare-proof-b}) we added and subtracted $v_{0}$ 
in the boundary terms; in (\ref{eq:poincare-proof-c}) we 
used the Cauchy-Schwarz inequality; in (\ref{eq:poincare-proof-e}) we used the definition 
 of the energy norm (\ref{eq:energy-norm}), the control for $v_0$ we have from \eqref{eq:v0-control}, and a standard 
 trace inequality for $\phi$ on $\Omega_{\delta,i}$;
and finally, in (\ref{eq:poincare-proof-f}) we used the regularity assumption (\ref{eq:poincare-dual-reg}). Choosing the data $\psi\in L^2(\Omega)$ to the dual problem as
\begin{align}
\psi=
\left\{
\begin{alignedat}{2}
&v_1 &\quad&\text{on $\Omega_{1,\delta}$}
\\
&v_2 &\quad&\text{on $\Omega_{2,\delta}\setminus\Omega_{1,\delta}$}
\\
&0 &\quad&\text{on $\Omega\setminus (\Omega_{1,\delta}\cup \Omega_{2,\delta})$}
\end{alignedat}
\right.
\end{align}
we have
\begin{align}
\| \psi \|_{\Omega}^2 = \| v_1 \|_{\Omega_{1,\delta}}^2 + \| v_2 \|_{\Omega_{2,\delta}\setminus\Omega_{1,\delta}}^2
\leq \sum_{i=1}^2 \| v_i \|_{\Omega_{i,\delta}}^2
\end{align}
and thus we obtain 
\begin{equation} \label{eq:poincare-patch}
\sum_{i=1}^2 \| v_i \|^2_{\mcT_{h,i}} \lesssim (1+\delta h^{\alpha/2-1}) \tn v \tn_h^2
\lesssim
\tn v \tn_h^2
\end{equation}
where we in the last inequality use $\alpha\geq 2$.

Finally, using \eqref{eq:sh0-inverseL2}, \eqref{eq:v0-control}, and $\alpha \geq 2$ we directly obtain a bound for the hybrid variable
\begin{align}
h^{-2} \| v_0 \|^2_{\mcT_{h,0}} 
\lesssim 
\| v_0 \|^2_{\partial\Omega_{1,\delta}} + \| v_0 \|^2_{s_{h,0}} 
\lesssim 
 \tn v \tn_h^2
\end{align}
which concludes the proof.
\end{proof}

\paragraph{Continuity and Coercivity.}
The form $A_h$ is continuous
\begin{align}\label{eq:cont}
\boxed{
A_h(v,w) \lesssim \tn v \tn_h \tn w \tn_h,  \qquad v,w \in W^e + W_h
}
\end{align} 
 and for $\beta > 0$ large enough coercive
 \begin{align}\label{eq:coer}
\boxed{ 
 \tn v \tn_h^2 \lesssim A_h(v,v),  \qquad v \in  W_h
 }
\end{align} 
The continuity follows from the Cauchy-Schwarz inequality and for the coercivity, we note that 
\begin{align}
A_h(v,v) &= \| v_0 \|^2_{s_{h,0}} + \sum_{i=1}^2 \|\nabla v_i \|^2_{\Omega_{i,\delta}} + \| v_i \|_{s_{h,i}}^2
\\&\nonumber \qquad\qquad\qquad\quad
- 2(\nabla_n v_i, v_i - v_0)_{\partial \Omega_{i,\delta}} 
+ \beta h^{-1} \| v_i - v_0 \|^2_{\partial \Omega_{\delta,i}}
\end{align}
and we can use the usual arguments provided the parameter $\beta$ is large enough.  

\paragraph{Interpolation.}
Before deriving the error estimates we recall some interpolation results.
By virtue of the patch extensions \eqref{eq:extension-patch} and interface extension \eqref{eq:extension-interface} the three fields of a function $v\in W^e$ is defined on the full mesh domains $\Omega_{h,0}$, $\Omega_{h,1}$, and $\Omega_{h,2}$.
We define an interpolation operator
\begin{align}
\pi_{h} : W^e \ni v=(v_0;v_1;v_2) \mapsto (\pi_{h,0}v_0;\pi_{h,1}v_1;\pi_{h,2}v_2) \in W_h
\end{align}
where $\pi_{h,i} : H^1(\Omega_{h,i}) \rightarrow V_{h,i}$ is the Scott-Zhang interpolation operator.  We 
choose the Scott-Zhang operator to preserve strong Dirichlet boundary conditions on $\partial \Omega$.

We now derive an interpolation estimate in the energy norm \eqref{eq:energy-norm}.
First, we consider the interpolation of the patch fields.
Combining standard interpolation error estimates and the stability of the extension operator we obtain
\begin{align} \label{eq:patch-interpolant}
\| ( I - \pi_{h,i} ) v_i^e \|_{H^m(\Omega_{i,\delta})} &\lesssim h^{p+1 - m} \| v_i \|_{H^{p+1}(\Omega_{i})}, \qquad m=0,1
\end{align} 
In the boundary terms in $\tn v \tn_h$ we separate the patch fields $v_i$ from the hybrid variable field $v_0$ using the triangle inequality, and then move $v_i$ onto $\Omega_{i,\delta}$ using a trace inequality.
The remaining patch field
term in $s_{h,i}$ can be directly estimated using elementwise trace inequalities and interpolation estimates. 
Next, we consider the interpolation of the hybrid variable field.
Similarly, as for \eqref{eq:patch-interpolant} we combine standard interpolation estimates with the stability of the extension operator and obtain
\begin{align}
\| ( I - \pi_{h,0} ) v_0^e \|_{H^m(\Omega_{h,0})} &\lesssim h^{p+2 - m} \| v_0 \|_{H^{p+1}(\Gamma)}, \qquad m=0,1
\end{align}
On the boundary terms, we apply an elementwise trace inequality to move onto $\Omega_{h,0}$ and then apply the above estimate.
What remains is to estimate the $s_{h,0}$-norm, where the first term from \eqref{eq:blockstab} is estimated
\begin{align} \label{:eq:sh0-partII}
h^{-\alpha} \|\nabla_\Gamma^\perp  (I - \pi_{h,0}) u_0^e \|^2_{\Omega_{h,0}}
&\leq
h^{-\alpha}
\| ( I - \pi_{h,0} ) u_0^e \|_{H^1(\Omega_{h,0})}^2
\lesssim
h^{2p+2 - \alpha} \| u_0 \|_{H^{p+1}(\Gamma)}
\end{align}
which holds for $\alpha \leq 2$, and the second term is estimated analogously to the patchwise $s_{h,i}$-norm.
Combining these estimates we obtain
\begin{equation}\label{eq:interpol}
\boxed{\tn v - \pi_h v \tn_h \lesssim  h^p \Big( \| v \|_{H^{p+1}(\Omega)} +   \| v \|_{H^{p+1}(\Gamma)} \Big) }
\end{equation}

\paragraph{Error Estimate.} We are now ready to prove an error estimate in the energy norm. 

\begin{thm}[Energy Norm Error] \label{thm:energy}
For $\alpha = 2$, it holds
\begin{align}
\boxed{\tn u^e - u_h \tn_h \lesssim (h^p + h^{-1/2}\delta) \Big( \| u \|_{H^{p+1}(\Omega)} + \| u \|_{H^{p+1}(\Gamma)} + \| u_0^e \|_{W^{2}_\infty(\Omega\cap U_\delta(\Gamma))} \Big) } 
\end{align}
\end{thm}
\begin{proof} It follows from coercivity that 
\begin{align}
\tn \pi_h u^e - u_h \tn_h 
\lesssim \inf_{v \in W_h \setminus \{0\}}
\frac{A_h(\pi_h u^e - u_h,v)}{\tn v \tn_h}
\end{align}
and we need to estimate the numerator.  We have 
\begin{align}
A_h(\pi_h u^e - u_h,v)&=A_h(\pi_h u^e ,v) - A_h(u_h,v)
\\
&= A_h(\pi_h u^e ,v) - l_h(v)
\\
&= a_h(\pi_h u^e ,v) + s_{h,0}(\pi_h u_0^e,v) - l_h(v)
\\
&= \underbrace{a_h(\pi_h u^e - u ,v)}_{I} + \underbrace{s_{h,0}(\pi_h u_0^e,v)}_{II}  
+ \underbrace{ a_h(u^e,v)  - l_h(v) }_{III}
\end{align}
Here $I$ is estimated using continuity (\ref{eq:cont}) and the interpolation error estimate (\ref{eq:interpol}), 
\begin{align}
a_h(\pi_h u^e - u^e ,v) \lesssim \tn \pi_h u^e - u^e  \tn_h  \tn v \tn_h 
\end{align}
For $II$ we have
\begin{equation}\label{eq:II}
\| \pi_h u_0^e \|_{s_{h,0}} = \| (\pi_h - I) u_0^e \|_{s_{h,0}} \lesssim h^p \| u_0 \|_{H^{p+1}(\Gamma)}
\end{equation}
where we, without affecting the value, can subtract $u_0^e$ in the $s_{h,0}$-norm since we for the first term in the norm have
\begin{align}
h^{-\alpha} \|\nabla_\Gamma^\perp  \pi_{h,0} u_0^e \|^2_{\Omega_{h,0}} 
&=
h^{-\alpha} \|\nabla_\Gamma^\perp  (\pi_{h,0} - I) u_0^e \|^2_{\Omega_{h,0}}
\end{align}
as the extension $u_0^e$ is constant orthogonal to $\Gamma$,
and the second term is defined in terms of jumps over mesh edges, which are zero for sufficiently regular $u_0^e$.
For $III$ we use partial integration 
\begin{align}
III &=  \sum_{i=1}^2 (\nabla u_i^e, \nabla v_i)_{\Omega_{i,\delta}} - (\nabla_n u_i^e, v_i - v_0)_{\partial \Omega_{i,\delta}} 
\\
&\qquad \quad
- (u_i^e - u_0^e, \nabla_n v_i)_{\partial \Omega_{i,\delta}} 
+ \beta h^{-1} ( u_i^e - u_0^e, v - v_0)_{\partial \Omega_{\delta,i}}
- (f_i^e,v_i )_{\Omega_{\delta,i}}
\\
&=  \sum_{i=1}^2 \underbrace{-(\Delta u_i^e, v_i)_{\Omega_{i,\delta}} 
+ (\nabla_n u_i^e, v_i )_{\partial \Omega_{i,\delta}} 
- (\nabla_n u_i^e, v_i - v_0)_{\partial \Omega_{i,\delta}} - (f_i^e,v_i )_{\Omega_{\delta,i}}}_{=\sum_{i=1}^2 (\nabla_n u_i^e, v_0)_{\partial \Omega_{i,\delta}}}
\\
&\qquad \quad
- (u_i^e - u_0^e, \nabla_n v_i)_{\partial \Omega_{i,\delta}} 
+ \beta h^{-1} ( u_i^e - u_0^e, v_i - v_0)_{\partial \Omega_{\delta,i}}
\\
&= \sum_{i=1}^2 (\nabla_n u_i^e, v_0)_{\partial \Omega_{i,\delta}} 
 - (u_i^e - u_0^e, \nabla_n v_i)_{\partial \Omega_{i,\delta}} 
+ \beta h^{-1} ( u_i^e - u_0^e, v_i - v_0)_{\partial \Omega_{\delta,i}}
\\
&\leq \Big|  \sum_{i=1}^2 (\nabla_n u_i^e, v_0)_{\partial \Omega_{i,\delta}}  \Big|
\\
&\qquad + \sum_{i=1}^2 \|u_i^e - u_0^e \|_{\partial \Omega_{i,\delta}}  \| \nabla_n v_i\|_{\partial \Omega_{i,\delta}} 
+ \beta h^{-1} \| u_i^e - u_0^e \|_{\partial \Omega_{i,\delta}} \|v_i - v_0 \|_{\partial \Omega_{\delta,i}}
\\ \label{eq:xyz}
&\lesssim \underbrace{\Big|  \sum_{i=1}^2 (\nabla_n u_i^e, v_0)_{\partial \Omega_{i,\delta}}  \Big|}_{III_1}
+ \underbrace{\Big( \sum_{i=1}^2 h^{-1/2} \|u_i^e - u_0^e \|_{\partial \Omega_{i,\delta}}\Big)}_{III_2}  \tn v \tn_h
\end{align}
To estimate $III_1$,  we add and subtract $v_{0}^e$ and utilize the interface condition \eqref{eq:model-problem-interface} to insert  $0=\sum_{i=1}^2 ((\nabla_{n_i} u_i)|_\Gamma)^e$, where the implied extension is $E_0$, in the second 
term, 
\begin{align}
 &\sum_{i=1}^2 (\nabla_n u_i^e, v_0)_{\partial \Omega_{i,\delta}} =  \sum_{i=1}^2 (\nabla_n u_i^e , v_0 - v_{0}^e)_{\partial \Omega_{i,\delta}} 
 +   \sum_{i=1}^2 (\nabla_n u_i^e - ((\nabla_n u_i)|_\Gamma)^e  ,  v_{0}^e)_{\partial \Omega_{i,\delta}}
\\
&\lesssim  \sum_{i=1}^2 \|\nabla_n u_i^e\|_{\partial \Omega_{i,\delta}}  \|v_0 - v_{0}^e\|_{\partial \Omega_{i,\delta}} 
 +   \sum_{i=1}^2 \| \nabla_n u_i^e - ( (\nabla_n u_i)|_\Gamma )^e \|_{\partial \Omega_{i,\delta}}  \|v_{0}^e\|_{\partial \Omega_{i,\delta}}
\\
 &\lesssim  \sum_{i=1}^2 \|u_i^e \|_{H^2 (\Omega_{i,\delta})} \delta h^{\alpha/2 -1} \| v_0 \|_{s_{h,0}}  
 +   \sum_{i=1}^2\delta \|  u  \|_{ W^2_\infty( \Omega  \cap U_\delta(\Gamma))}  \|v_{0}^e\|_{\partial \Omega_{i,\delta}}
\end{align}
In the last inequality, we utilize \eqref{eq:v0-pert} for the first term and a Taylor argument for the second term. We then use $\|v_{0}^e\|_{\partial \Omega_{i,\delta}} \lesssim h^{-1} \|v_{0}\|_{\mcT_{h,0}}$ and the Poincar\'e inequality \eqref{eq:poincare} to bound the test function in terms of the energy norm.
Finally,  for $III_2$ we have using similar estimates
\begin{align}
h^{-1} \|u^e - u_0^e \|^2_{\partial \Omega_{i,\delta}} \lesssim h^{-1} \delta^2 \| u \|^2_{W^1_\infty(U_\delta(\Gamma))}
\end{align}
and thus the proof is complete.
\end{proof}

\begin{rem}[Analysis in the Isogeometric Case]\label{rem:iga-ext}
%out-of-plane gaps, kinks, surfaces, cutfem
To extend the analysis above to the isogeometric multipatch case, we should consider the following aspects:%, most of which are covered in \cite{JonLarLar17}:
\begin{itemize}
\item \emph{Surface patches.}
As described in Remark~\ref{rem:extension2isogeometry}, extending the method to the isogeometric case, i.e., parametrically described surface patches, is done by transforming the terms back to a two-dimensional reference domain as in \eqref{eq:ahi-surf}. In the reference domain, the problem is a standard problem with variable coefficients given by the metric tensor, and this can easily be included in the analysis by assuming suitable bounds on the patch parametrization. See \cite{JonLarLar17} for details. For a complete analysis, we should also consider parametrization errors yielding an approximate metric tensor, which can be handled using a Strang type argument, see \cite{MR3903205} for a similar situation.

\item \emph{Sharp edges and out-of-plane gaps.}
Cases when the gap occurs where the patches coupled over the interface do not lie on the same smooth surface are covered by the above analysis thanks to the use of the three-dimensional hybrid variable. This can be seen by reviewing the proof of Lemma~\ref{lem:hybrid-variable-control} where it is not necessary that the exact interface $\Gamma$ is placed in the same plane as each patch. Further, since the interface conditions in the surface case \eqref{eq:surf-interface-cond} are formulated such that sharp edges are allowed, this in itself poses no additional difficulty in the analysis. See also \cite{HanJonLar17}, where formulations for problems on surfaces with sharp edges are developed and analyzed. 

%, for instance when the exact geometry is not $C^1$ over the interface, or the interface connects more than two patches,

% at the interface are actually covered by the above analysis thanks to the use of the three-dimensional hybrid variable control
%What might seem...
%Notably... agnostic...

\item \emph{Trimmed reference domains.} Since we in the method \eqref{eq:method} allow for trimmed patches, such that the computational mesh for each patch is not required to conform to its reference domain geometry, a complete analysis should also include the ghost penalty stabilization terms discussed in Remark~\ref{rem:sh_i}. Note that in the isogeometric multipatch situation, it is natural to append this stabilization in the reference domain, see \cite{JonLarLar17}.
\end{itemize}

\end{rem}

\section{Numerical Experiments}

\paragraph{Implementation.}

The method was implemented in MATLAB, largely following the details presented in \cite{JonLarLar17,HanJonLar17}. This implementation utilizes the available parametric mappings in the surface description, where patchwise surface terms are pulled back to a two-dimensional reference domain before assembly. An upshot of the hybridized approach is that the assembly of the interface terms is done patchwise, such that no knowledge about other patches on the other side of the interface is required. Hence, there is also no need of finding the corresponding point in adjacent patches, which can be cumbersome to do efficiently and robustly since it involves the inverse of surface mappings -- in particular when the interfaces are not exact.

A new component for this work is the implementation of the hybrid variable, which entails the construction of its approximation space and the assembly of the hybrid variable stabilization. In our implementation, the hybrid variable mesh is extracted from a three-dimensional structured background hexahedral grid, by traversing all patch boundaries without boundary conditions and marking elements passed in the background grid, and we equip this mesh with a continuous approximation space.
%This mesh is then equipped with full regularity tensor product B-spline basis functions of degree $p$.
The assembly of the stabilization includes evaluation of $\nabla_\Gamma^\perp = (I - t_\Gamma \otimes t_\Gamma)\nabla_{\IR^3}$, the part of the gradient normal to the (artificial) interface $\Gamma$. We base our implementation on interpolation of $t_\Gamma\otimes t_\Gamma$ using tensor product Lagrange elements of degree $p$, where the value for $t_\Gamma$ at each interpolating point is set to the tangent value at the nearest closest point on the patch boundaries. Different approaches to this assembly are outlined in Remark~\ref{rem:hybrid-stab}, and we believe that, in practice, an octree-based mesh structure in combination with a projection-based method for extending the tangential field $t_\Gamma$, avoiding the closest point mapping, would give the most flexible, efficient and robust implementation.

For the import of CAD geometries into MATLAB we utilize IGES Toolbox~\cite{IGES}, which allows us to import CAD surfaces in IGES format. This gives us a set of patches described via NURBS as well as a set of trim curves. To evaluate the NURBS and its derivatives we use the NURBS toolbox\cite{NURBS}.

\paragraph{Parameter Choices and Approximation Spaces.}
As described above we cover all patch boundaries corresponding to interfaces with a structured hexahedral mesh with global mesh size $h$, where typically $h \geq \delta$, which is the mesh for the hybrid variable. We equip each surface patch $\Omega_{i,\delta}$ with a structured quadrilateral mesh in the two-dimensional reference domain, covering $\widehat\Omega_{i,\delta}$, where the mesh size in the reference domain is chosen such that the mapped elements on the surface approximately have size $h$. On each mesh, we define an approximation space using full regularity tensor product B-splines of degree $p$, where $p=2$ unless otherwise stated.
%This means that we have separate approximation spaces on each patch, and that the hybrid variable approximation space is guaranteed continuous in simply connected parts of the hybrid variable mesh.
For the Nitsche penalty parameter we use $\beta = 25 p^2$ and for the stabilization parameters we use $\tau_0=\tau_i=0.01$.

\paragraph{Convergence Studies.}
As a model problem for our quantitative studies we consider the Laplace-Beltrami problem $-\Delta_\Omega u = f$ with non-homogeneous Dirichlet boundary conditions. We construct a sequence of surface domains with a gap, where we can vary the gap size $\delta$, and which is illustrated in Figure~\ref{fig:surface-model-problem}.
Specifically, we map the unit square onto the surface of a torus, where the unit square has the partition
\begin{align}
\widehat\Omega_{1,\delta} = \{ (\hatx,\haty) : 0 < \hatx,\haty < 1 \,;\ \hatx^2+ \haty^2 > 0.9^2  \}
\,,\
\widehat\Omega_{2,\delta} = \{ (\hatx,\haty) : \hatx^2+ \haty^2 < 0.9^2  \}
\end{align}
where $\widehat\Omega_{2,\delta}$ is an inner disc and $\widehat\Omega_{1,\delta}$ is the remaining outer part.
We map these reference domains onto the surface, such that $\Omega_i = F_i(\widehat\Omega_{i,\delta})$, using the mappings
\begin{align}
F_1(\hatx, \haty) &=  [(R + r \cos{\scriptstyle\frac{5\pi\hatx}{3}} ) \cos{\scriptstyle\frac{5\pi\haty}{18}},\,
 (R + r \cos{\scriptstyle\frac{5\pi\hatx}{3}}) \sin{\scriptstyle\frac{5\pi\haty}{18}},\,
 r \sin{\scriptstyle\frac{5\pi\hatx}{3}} ]
\\
F_2(\hatx, \haty) &= F_1(\hatx, \haty)
+ \delta [\cos{\scriptstyle\frac{5\pi}{6}} \cos{\scriptstyle\frac{5\pi}{36}},\, \cos{\scriptstyle\frac{5\pi}{6}} \sin{\scriptstyle\frac{5\pi}{36}},\, \sin{\scriptstyle\frac{5\pi}{6}} ]
\label{eq:param_torus}
\end{align}
where we note that the latter mapping is shifted a distance $\delta$ in the normal direction of the disc midpoint.
We manufacture a problem on the exact ($\delta=0$) surface with known analytical solution $u=\sin(3x)\sin(3y)\sin(3z)$. This ansatz is a restriction of a function of three-dimensional Cartesian coordinates to the surface, and to evaluate the data $f=-\Delta_\Omega u$ we express the Laplace-Beltrami operator
$\Delta_\Omega u = \Delta_{\IR^3} u - \partial_{nn} u - 2H\partial_n u$
where $\Delta_{\IR^3}$ is the three-dimensional Laplacian, $\partial_{n}$ and $\partial_{nn}$ are the first and second order derivatives in the direction of the surface normal $n$, and $H$ is the mean curvature of the surface.
When measuring the error in the experiments below, we on the shifted patch $\Omega_{2,\delta}$ lift the analytical solution from the exact surface using the closest point mapping of the torus.

In the standard situation we foresee, the gap is caused by the finite precision in the parameterization of the trim curves in the CAD description, meaning that the gap size $\delta$ is fixed with respect to the mesh size $h$. Convergence results for the model problem with various sizes of a fixed gap are presented in Figure~\ref{fig:fixed-gap-convergence}. As expected, we note optimal order convergence until the error levels out due to the geometric error induced by the gap, where a smaller gap size gives a smaller lower bound on the error.

To give some validation to our error estimate in Theorem~\ref{thm:energy}, we in Figure~\ref{fig:gap-scaling-convergence} also consider the convergence of the model problem where the gap size $\delta$ is scaled by the mesh size $h$ to various powers. We note that gap size scalings of $\delta \sim h^p$ and $\delta \sim h^{p+1}$ seem to be needed to achieve optimal order convergence in $H^1$-seminorm and $L^2$-norm, respectively. The former result is $h^{1/2}$ better than would be expected considering the energy norm bound in Theorem~\ref{thm:energy}. We believe our estimate to be sharp and that the reason for this discrepancy is that the $H^1$-seminorm on the patches is in fact better than the full energy norm that also includes the interface terms. We will return to the analysis of this in another contribution.

\begin{figure}
\centering	
\begin{subfigure}[t]{0.35\linewidth}\centering
\includegraphics[width=0.9\linewidth]{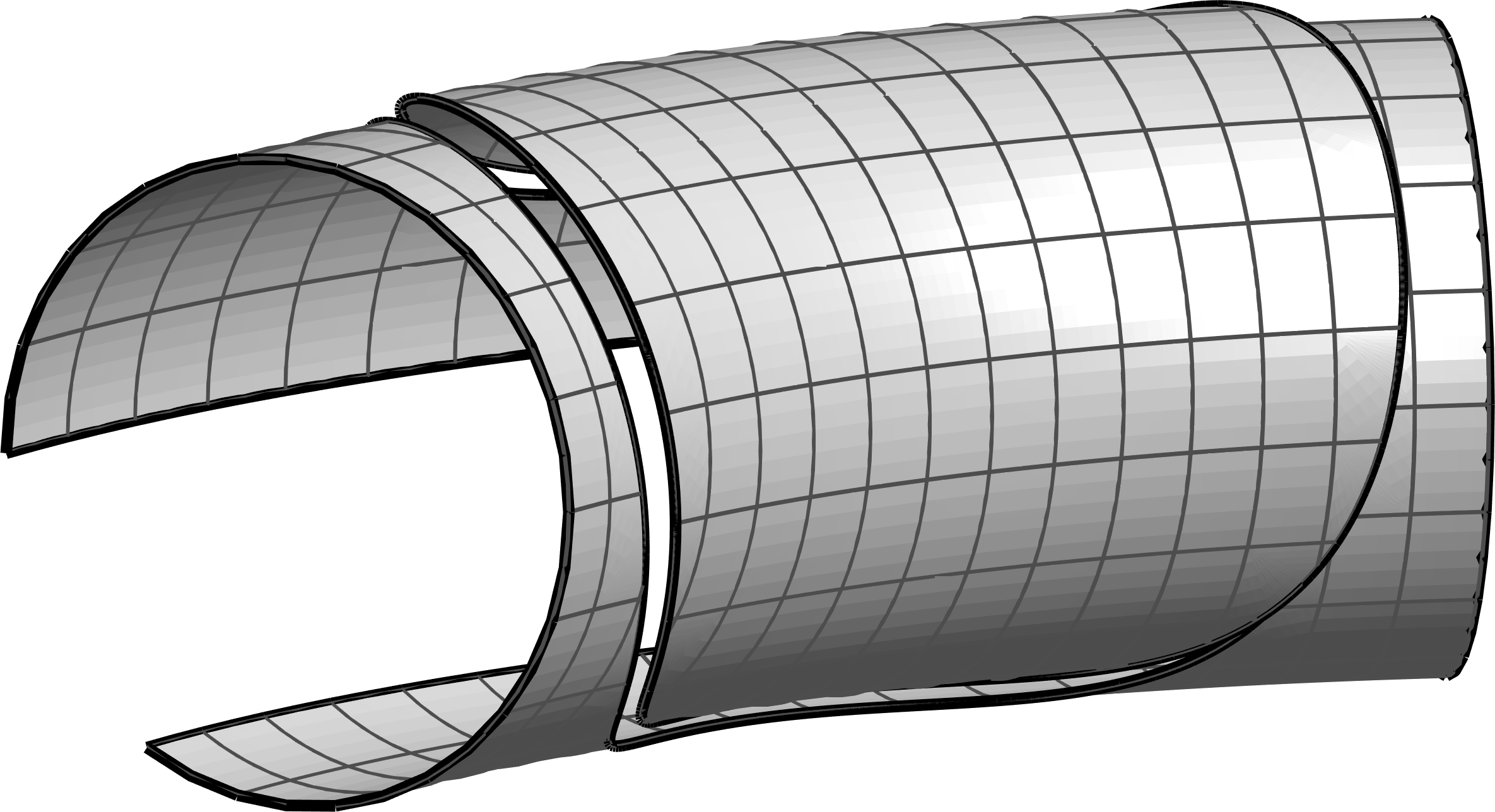}
\subcaption{Perturbed surface domain}
\label{fig:cad}
\end{subfigure}
\begin{subfigure}[t]{0.35\linewidth}\centering
\includegraphics[width=0.9\linewidth]{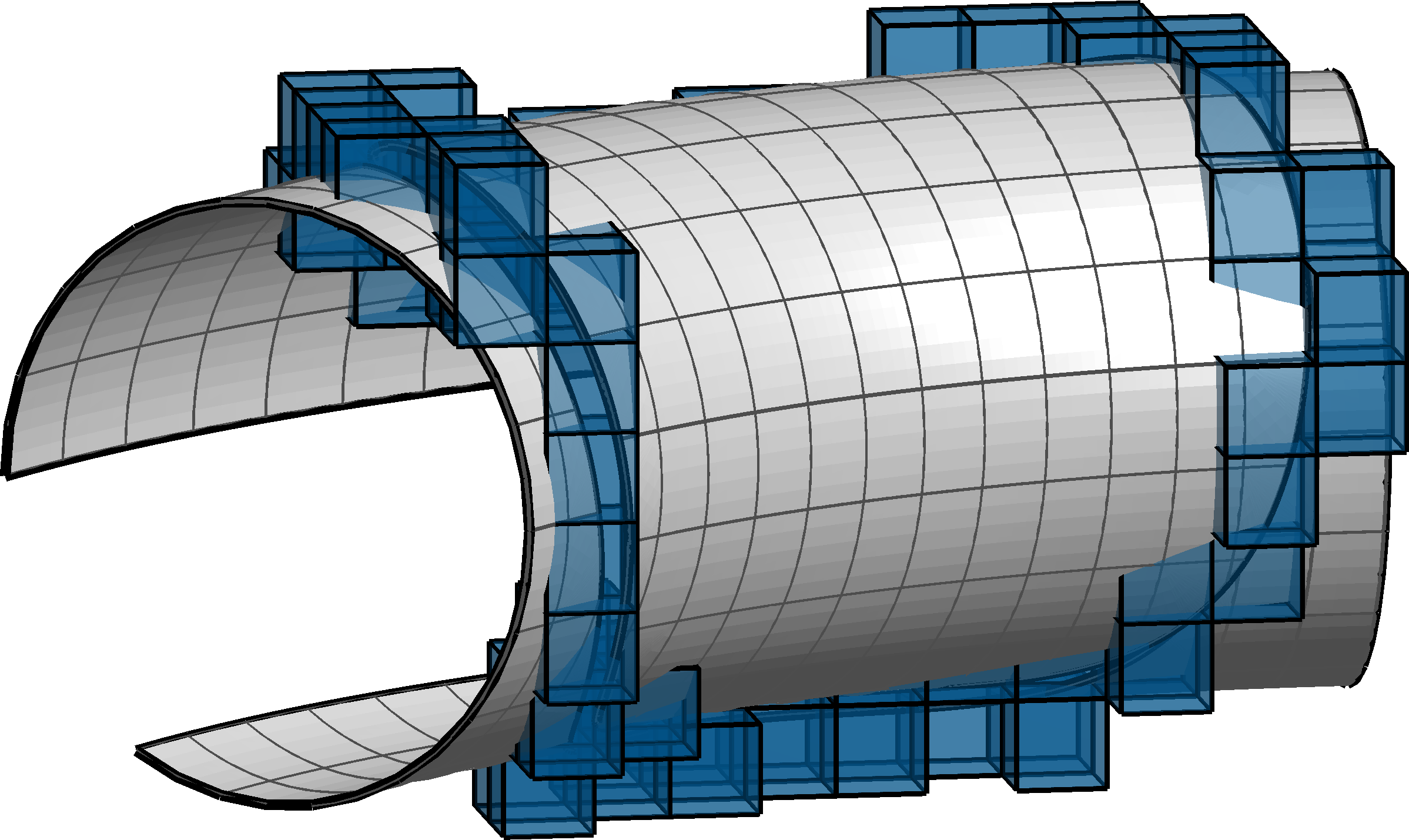}
\subcaption{Hybrid variable mesh}
\label{fig:physdom}
\end{subfigure}
\caption{\emph{Surface model problem.} A sequence of model geometries is constructed by decomposing the unit square into two parts by cutting away a circle, mapping the two parts onto the surface of a torus, and shifting the inner part (disc) a distance $\delta$ in the normal direction at its midpoint. The resulting domain (a) features a gap where the direction of the gap varies over the interface, where in some parts the gap is mainly in the tangential plane of the torus surface and some parts are mainly normal to the surface.
In (b) the hybrid variable mesh is shown, which covers both sides of the interface and is extracted from a uniform background grid.
}
\label{fig:surface-model-problem}
\end{figure}

\begin{figure}
\centering	
\begin{subfigure}{0.35\linewidth}\centering
\includegraphics[width=0.9\linewidth]{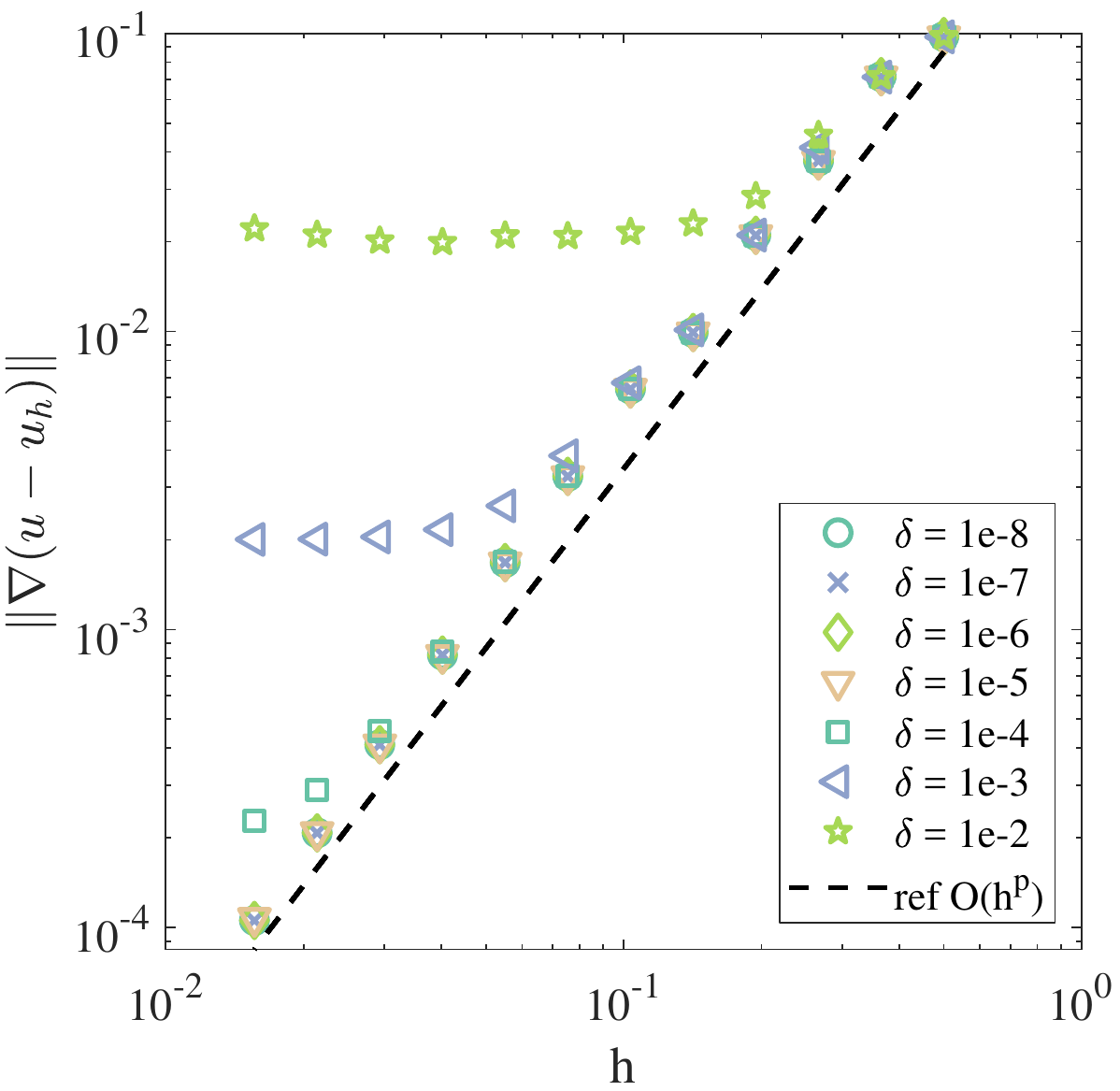}
\subcaption{Error in $H^1$-seminorm}
\end{subfigure}
\begin{subfigure}{0.35\linewidth}\centering
\includegraphics[width=0.9\linewidth]{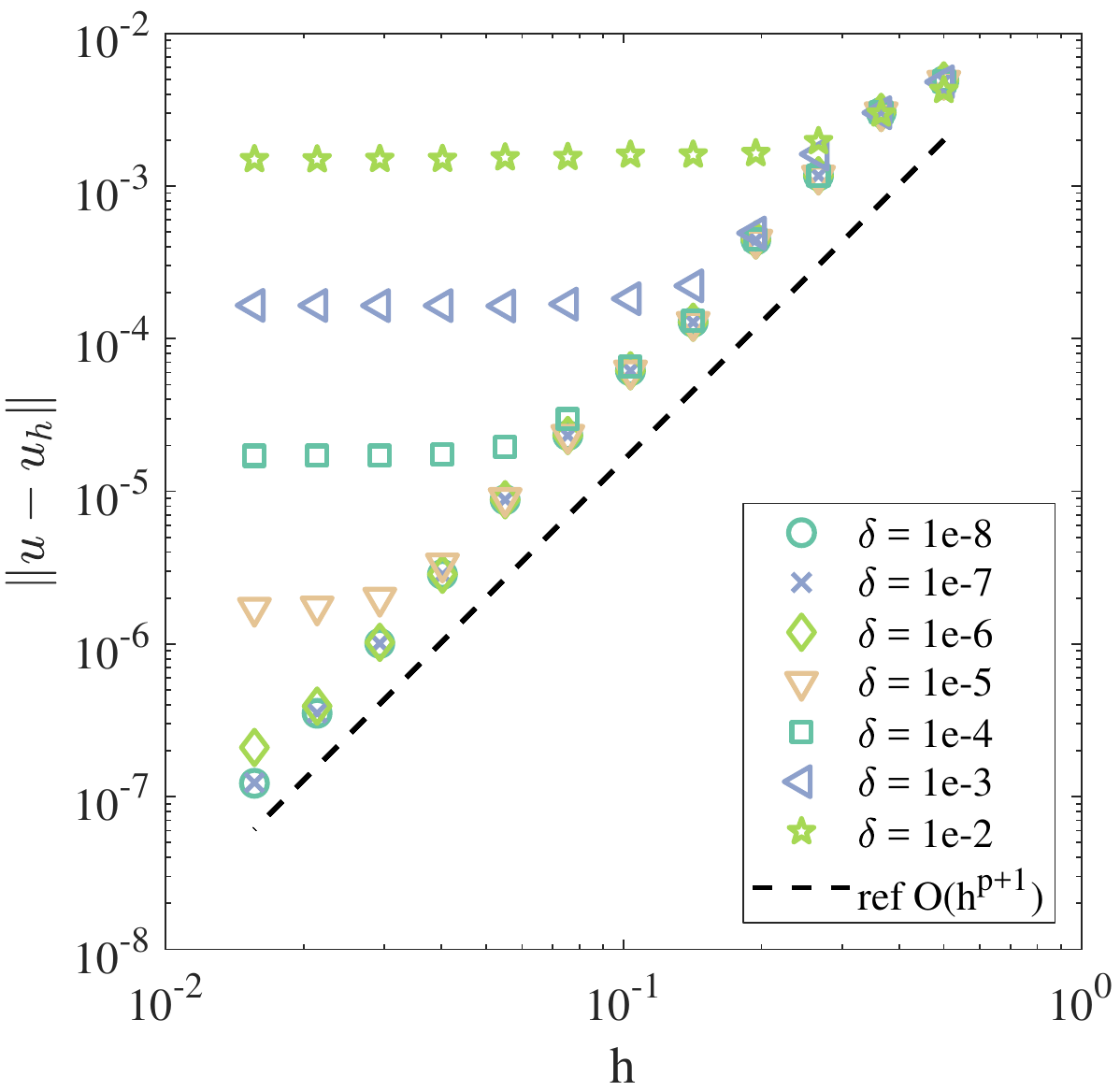}
\caption{Error in $L^2$-norm}
\end{subfigure}
\caption{\emph{Convergence for a fixed gap size.} In most practical situations the gap size $\delta$ is not something we can choose but is rather given by a provided CAD surface. Here, we consider convergence in the surface model problem for a number of different fixed gap sizes $\delta$ when discretizing using full regularity B-spline basis functions of degree $p=2$. We note optimal order convergence until the error eventually levels out due to the geometric error.}
\label{fig:fixed-gap-convergence}
\end{figure}

\begin{figure}
\centering	
\begin{subfigure}{0.3\linewidth}\centering
\includegraphics[width=0.95\linewidth]{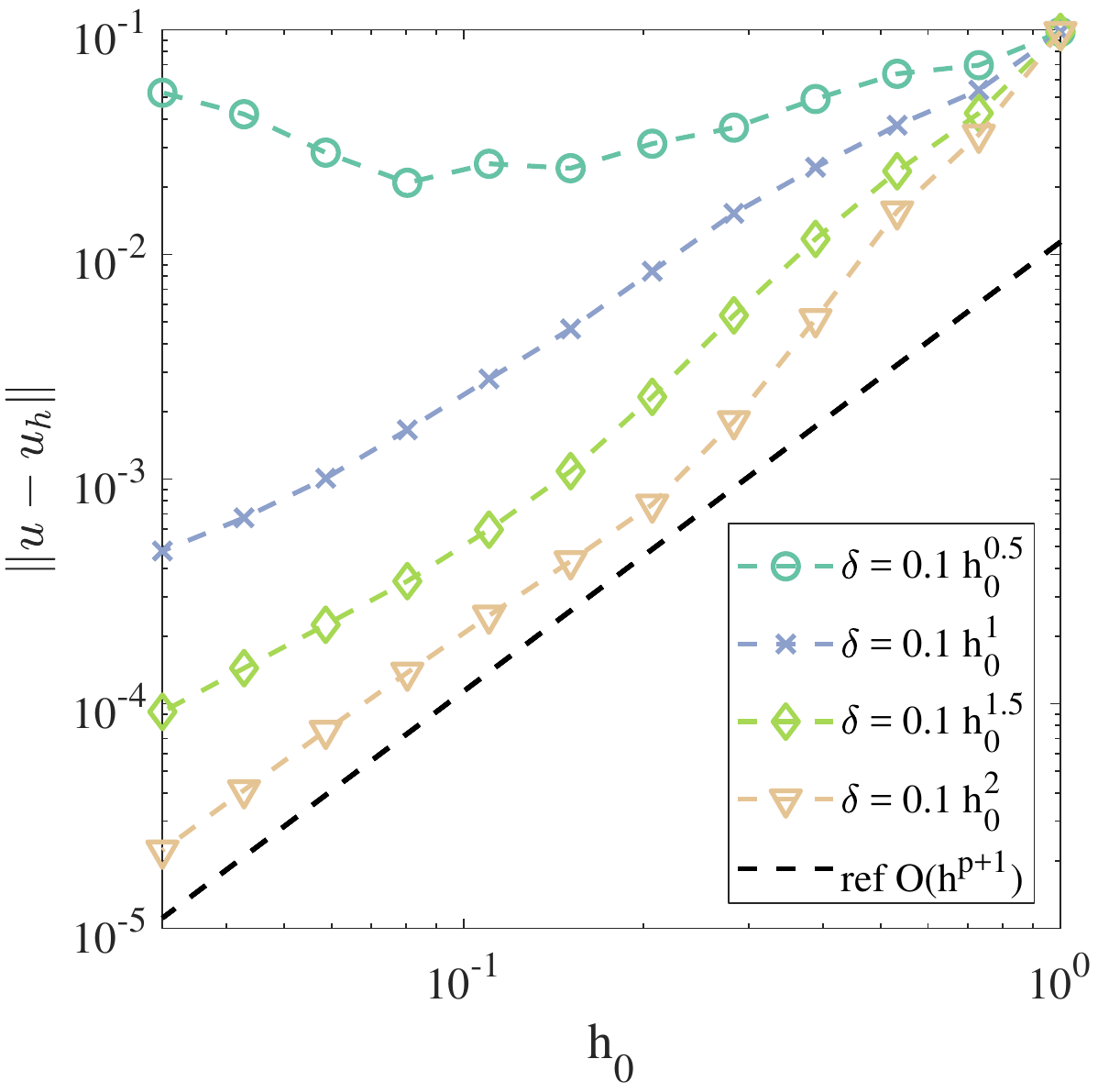}
\caption{$p=1$}
\label{fig:torus_L2_q1_alpha}
\end{subfigure}
\begin{subfigure}{0.3\linewidth}\centering
\includegraphics[width=0.95\linewidth]{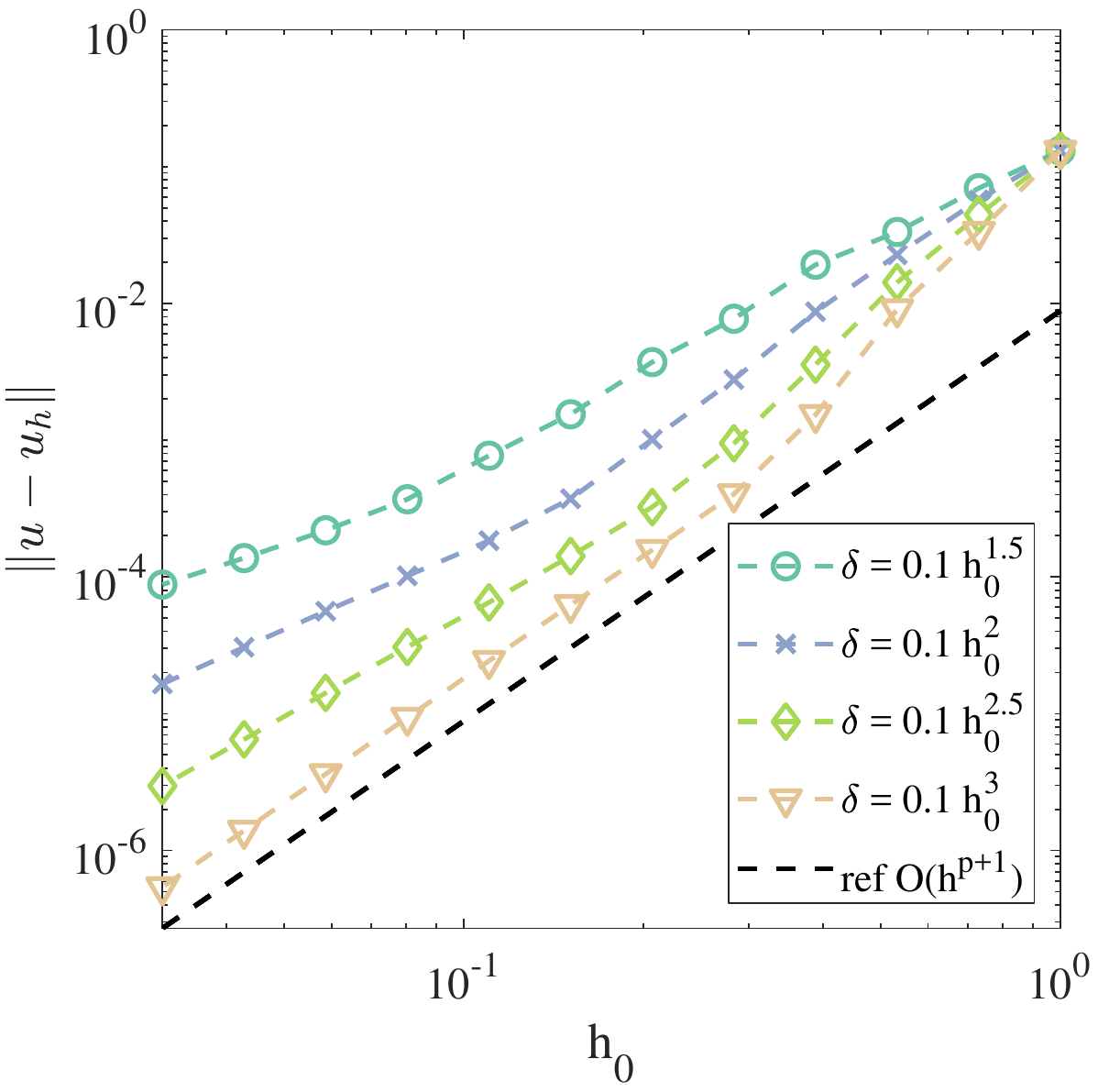}
\caption{$p=2$}
\label{fig:torus_L2_q2_alpha}
\end{subfigure}
\begin{subfigure}{0.3\linewidth}\centering
\includegraphics[width=0.95\linewidth]{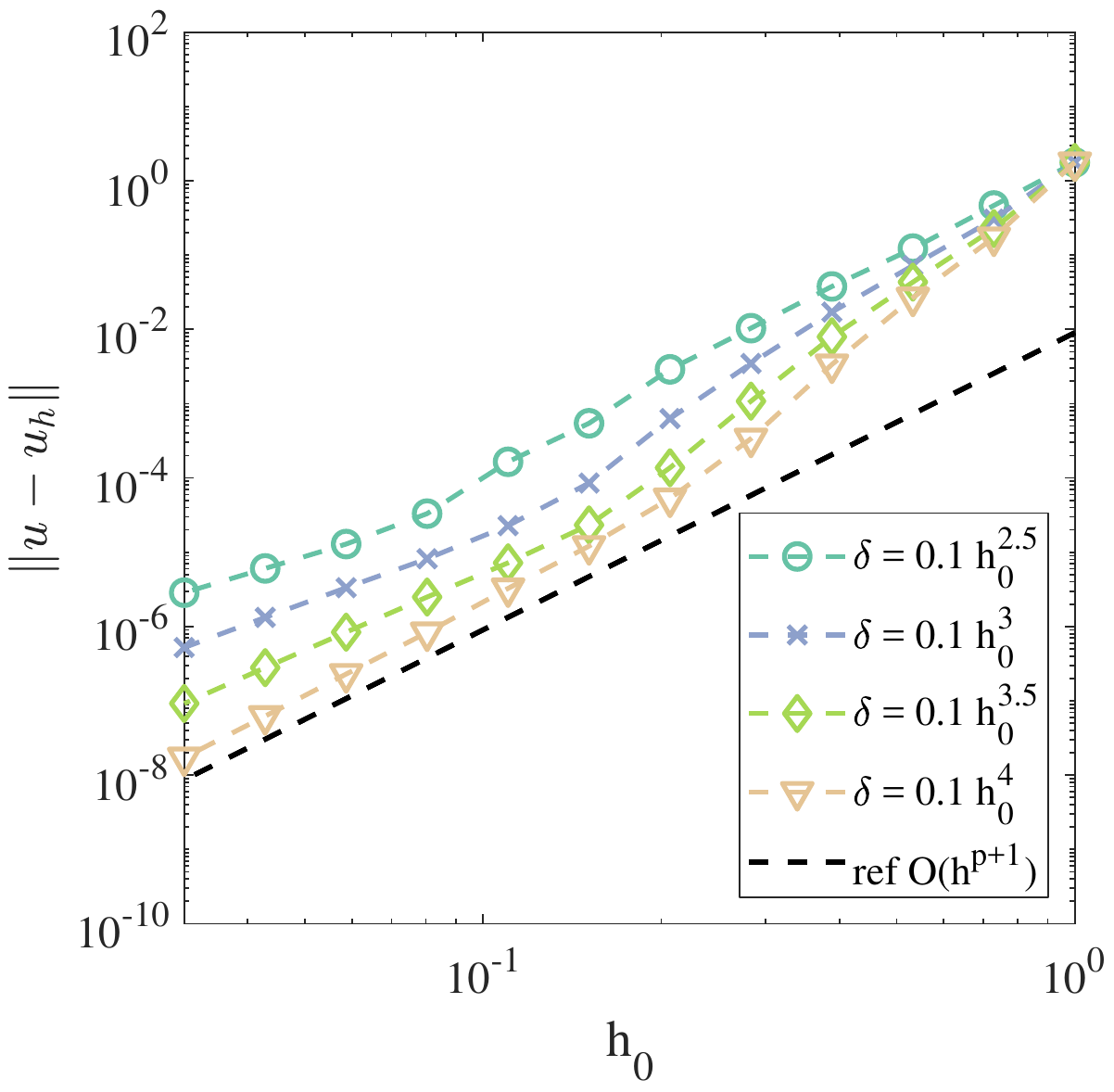}
\caption{$p=3$}
\label{fig:torus_L2_q3_alpha}
\end{subfigure}
\\[0.5em]
\begin{subfigure}{0.3\linewidth}\centering
\includegraphics[width=0.95\linewidth]{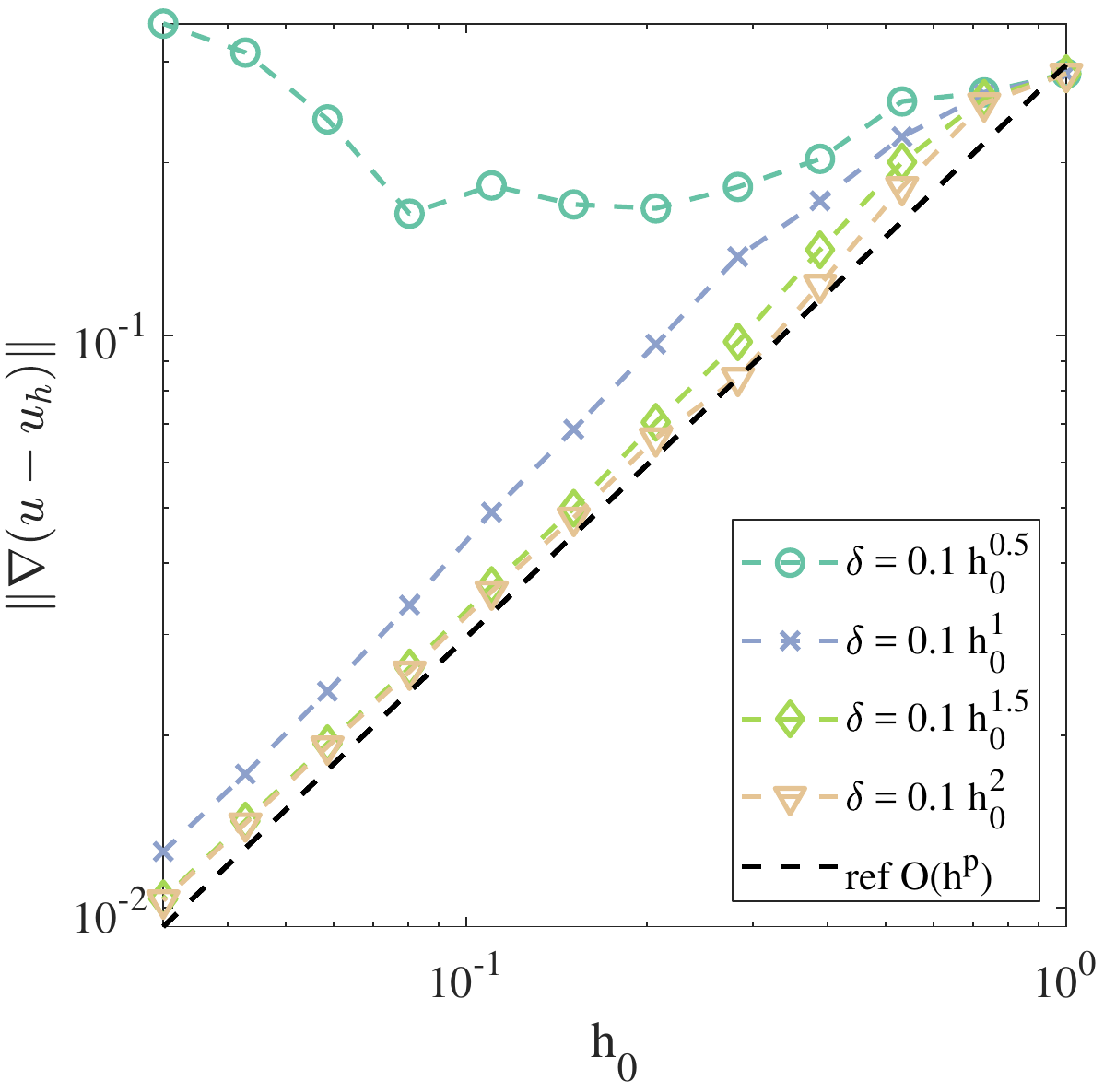}
\subcaption{$p=1$}
\label{fig:torus_H1_q1_alpha}
\end{subfigure}
\begin{subfigure}{0.3\linewidth}\centering
\includegraphics[width=0.95\linewidth]{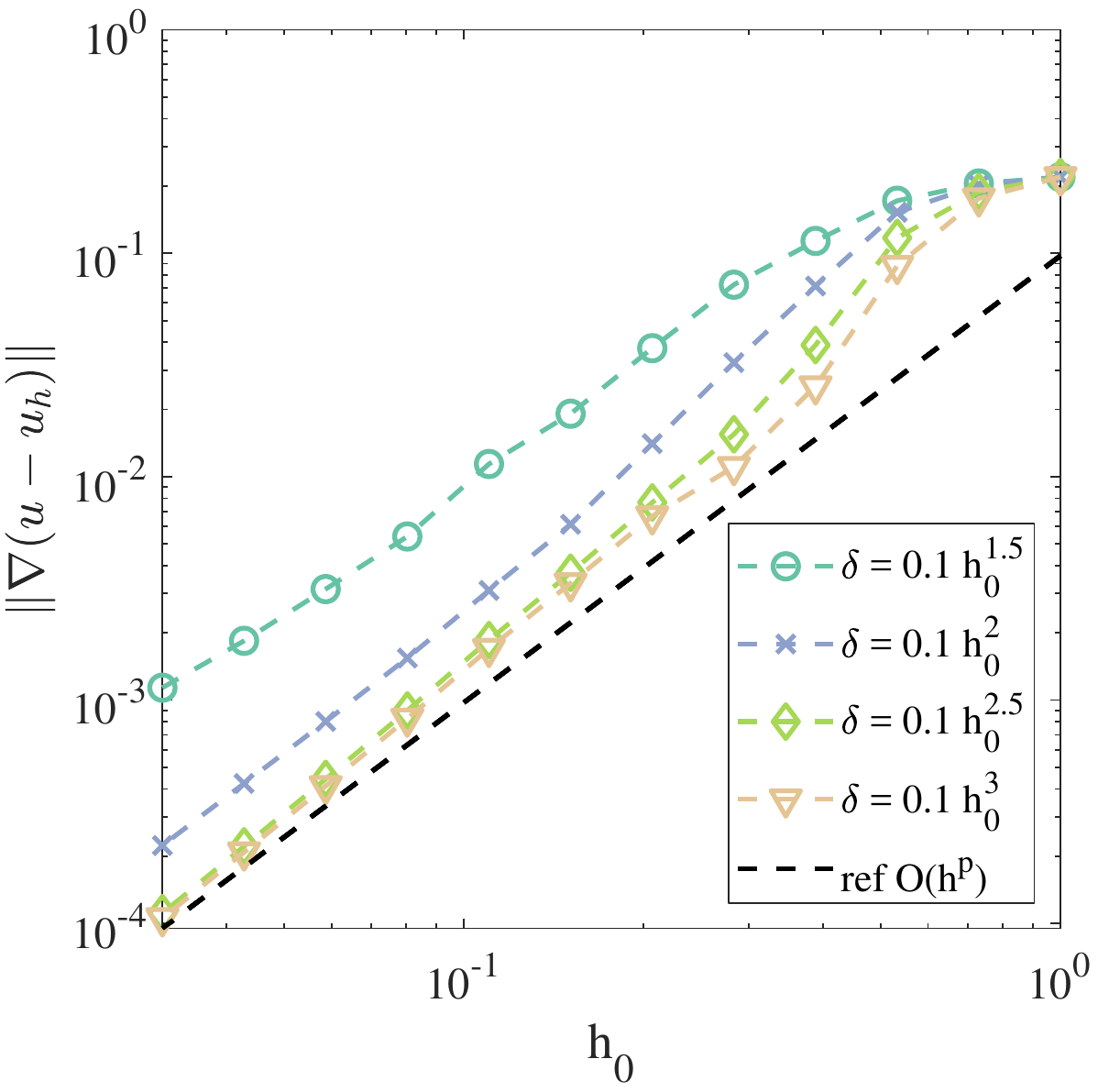}
\subcaption{$p=2$}
\label{fig:torus_H1_q2_alpha}
\end{subfigure}
\begin{subfigure}{0.3\linewidth}\centering
\includegraphics[width=0.95\linewidth]{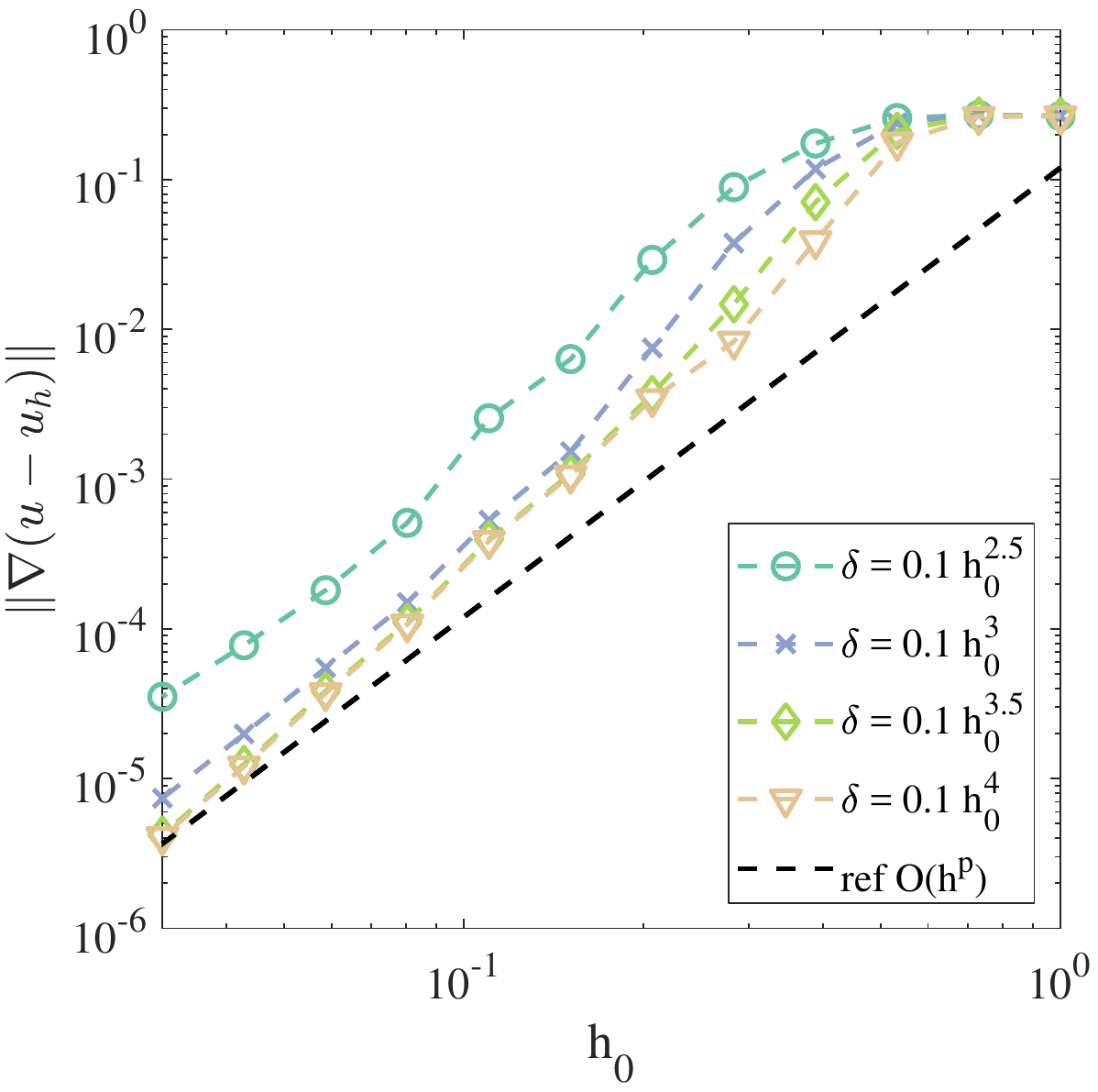}
\caption{$p=3$}
\label{fig:torus_H1_q3_alpha}
\end{subfigure}
\caption{\emph{Convergence with gap scaling.}
In this numerical study, we investigate how the gap size $\delta$ must scale with $h$ to maintain optimal order convergence. We scale the gap size for the surface model problem as $\delta = 0.1h_0^{s}$ for various values of $s$.
Here, $h_0$ is the mesh size $h$ normalized by the largest mesh size to have the same initial gap size independently of $s$.
We discretize using full regularity B-spline basis functions of degree $p$.
}
\label{fig:gap-scaling-convergence}
\end{figure}

\paragraph{Hybrid Variable Studies.}
Next, we study the behavior of the hybrid variable stabilization. To facilitate better visualizations of the numerical solution, including the hybrid variable, we construct a model problem in the two-dimensional plane by taking the unit square, cutting out a disc, and shifting this disc a distance $\delta$ in the plane causing a gap. While this model geometry is entirely defined in the two-dimensional plane the hybrid variable is still defined on a three-dimensional mesh covering the gap, so for visual clarity, we plot the hybrid variable solution along its intersection with the plane.

%\begin{itemize}
%\item 
Intuitively, the desired effect of the hybrid variable stabilization is to make the hybrid variable solution constant across the gap while being sufficiently weak not to affect the solution along either side of the gap. In Figure~\ref{fig:blockstab-weak-good-strong} we vary the strength of this stabilization in one gap situation and plot the hybrid variable solution. We note that a too weak stabilization causes the hybrid variable solution to vary significantly over the gap, while an apt stabilization as desired keeps the solution constant across the gap. On the other hand, a too strong stabilization induces looking due to the curved interfaces, which deteriorates the solution also along the gap. This illustrates the importance of choosing an accurate scaling of the hybrid variable stabilization.

%\item
In Figure~\ref{fig:increasing-gap} we look at how the patch error is qualitatively affected by the gap size $\delta$. Looking at the outer patch, whose location is constant with respect to the gap size, we as expected see that the error increases with the gap size. The hybrid variable stabilization seems to do its job since the hybrid variable solution keeps approximately constant across the gap for all gap sizes. Due to the way we extract the hybrid mesh in our implementation, there are, in the case of the largest gap, some elements missing in the region covering the gap. This, however, seems to have little influence on the hybrid variable solution, which is likely thanks to the extended support of the B-spline basis functions.

%\end{itemize}

\begin{figure}
\centering	
\begin{subfigure}[t]{0.3\linewidth}\centering
\includegraphics[width=0.9\linewidth]{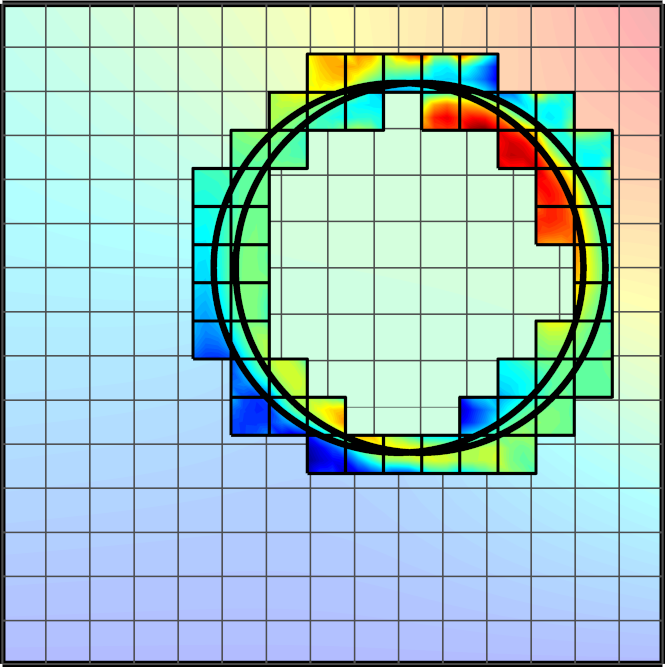}
\subcaption{Too weak stabilization}
\label{fig:square:blocksol:low}
\end{subfigure}
\begin{subfigure}[t]{0.3\linewidth}\centering
\includegraphics[width=0.9\linewidth]{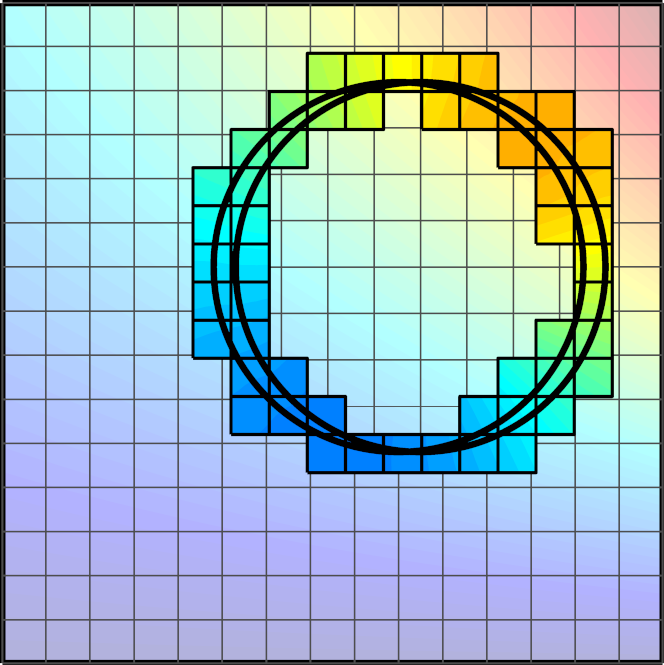}
\subcaption{Suitable stabilization}
\label{fig:square:blocksol:mid}
\end{subfigure}
\begin{subfigure}[t]{0.3\linewidth}\centering
\includegraphics[width=0.9\linewidth]{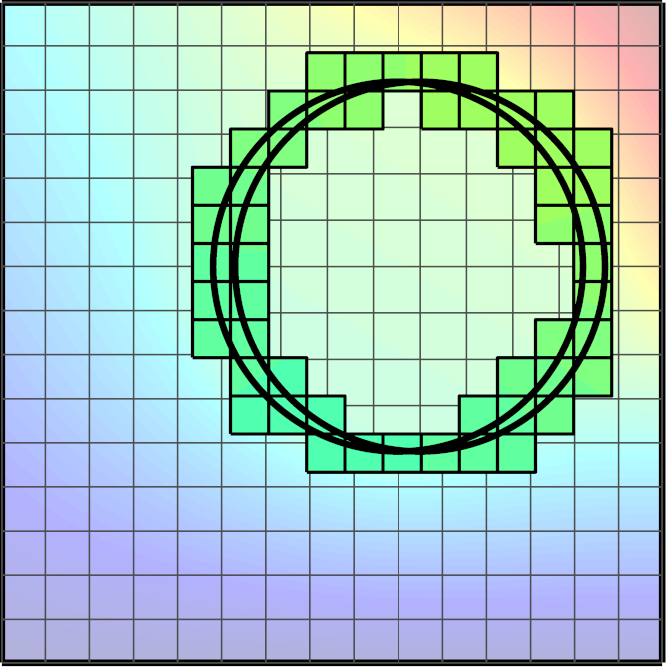}
\subcaption{Too strong stabilization}
\label{fig:square:blocksol:high}
\end{subfigure}
\caption{\emph{Effect of hybrid variable stabilization.} Numerical solution of the hybrid variable using a too weak, a suitable, respectively a too strong hybrid variable stabilization. For reference, the (faded) numerical solution in the patches is also presented.
In the case of a too weak stabilization as seen in (a) the hybrid variable varies substantially across the gap, whereas a suitable stabilization as in (b) yields the desired behavior where the hybrid variable is almost constant across the gap. Using a too strong stabilization, as in (c), comes with the risk of locking in the hybrid variable as the coupling between the normal and tangential components, induced by the curvature of the interface, may become dominant.
}
\label{fig:blockstab-weak-good-strong}
\end{figure}

\begin{figure}
\centering
\begin{subfigure}[t]{0.3\linewidth}\centering
\includegraphics[width=0.9\linewidth]{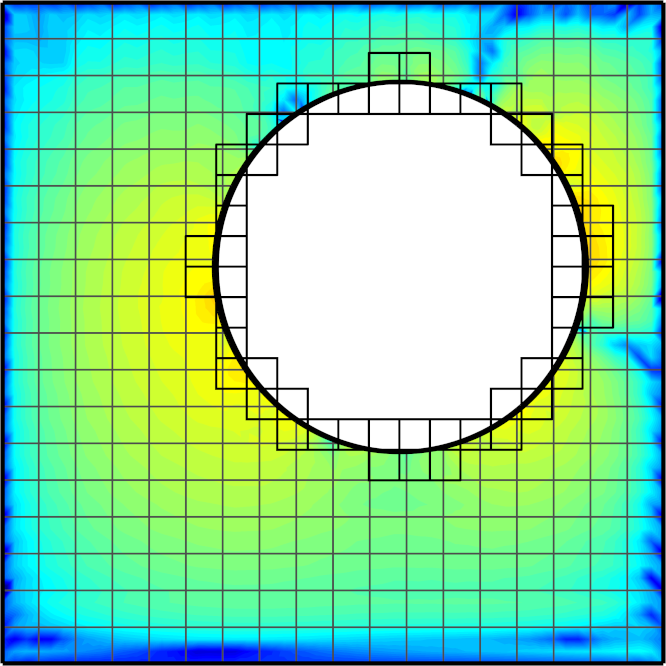}
\subcaption{$\log|u-u_h|$, $\delta h^{-1}=0.05$}
\label{fig:square:abserr:005}
\end{subfigure}
\begin{subfigure}[t]{0.3\linewidth}\centering
\includegraphics[width=0.9\linewidth]{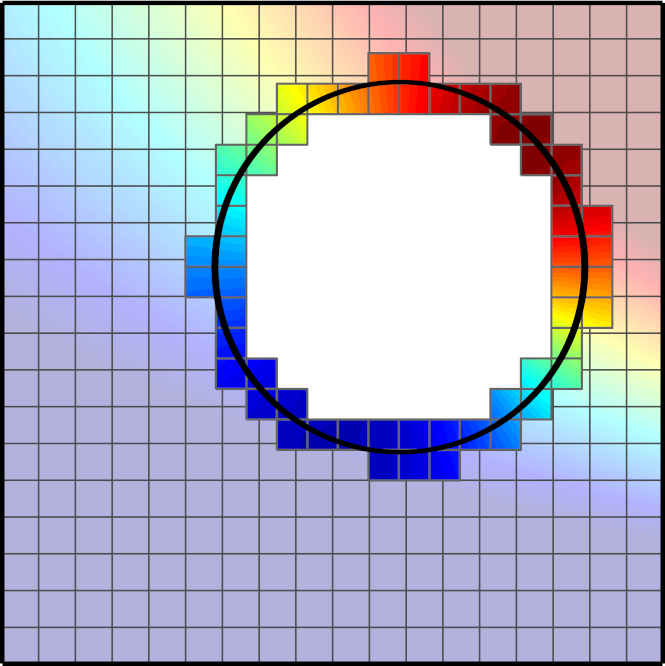}
\subcaption{$u_{h,0}$, $\delta h^{-1}=0.05$}
%\label{fig:square:blocksol:low}
\end{subfigure}
\\[0.5em]
\begin{subfigure}[t]{0.3\linewidth}\centering
\includegraphics[width=0.9\linewidth]{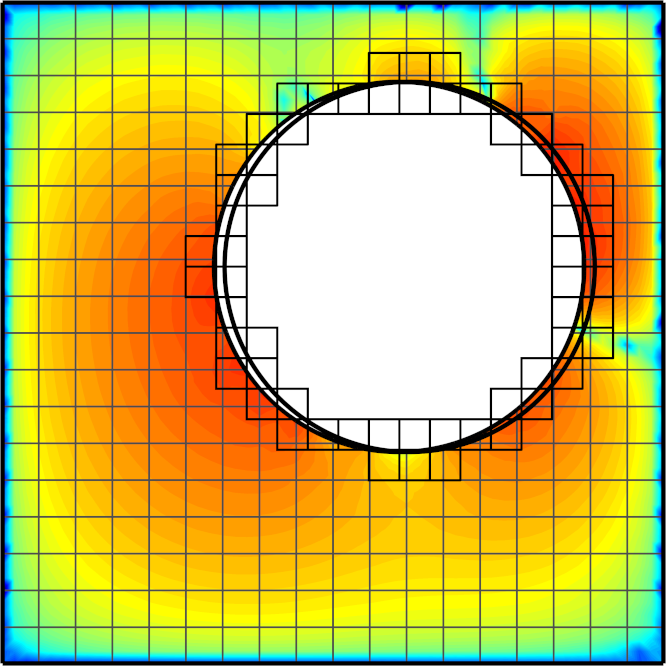}
\subcaption{$\log|u-u_h|$, $\delta h^{-1}=0.2$}
\label{fig:square:abserr:02}
\end{subfigure}
\begin{subfigure}[t]{0.3\linewidth}\centering
\includegraphics[width=0.9\linewidth]{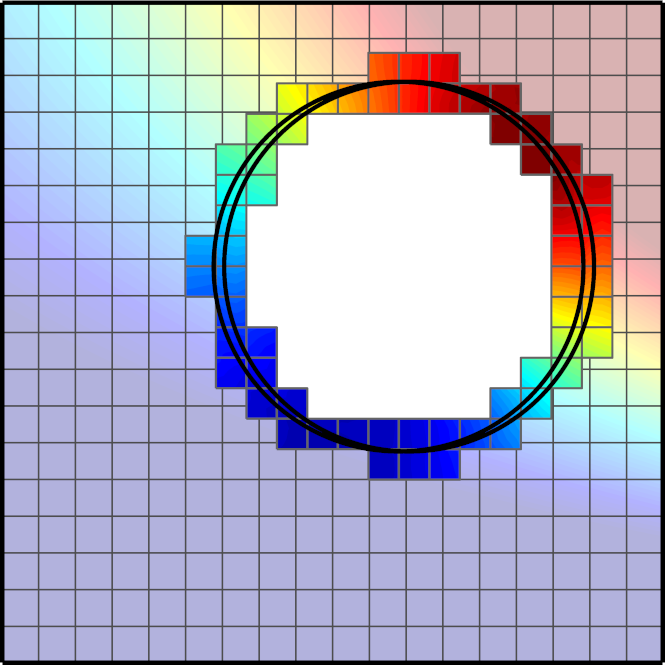}
\subcaption{$u_{h,0}$, $\delta h^{-1}=0.2$}
%\label{fig:square:blocksol:mid}
\end{subfigure}
\\[0.5em]
\begin{subfigure}[t]{0.3\linewidth}\centering
\includegraphics[width=0.9\linewidth]{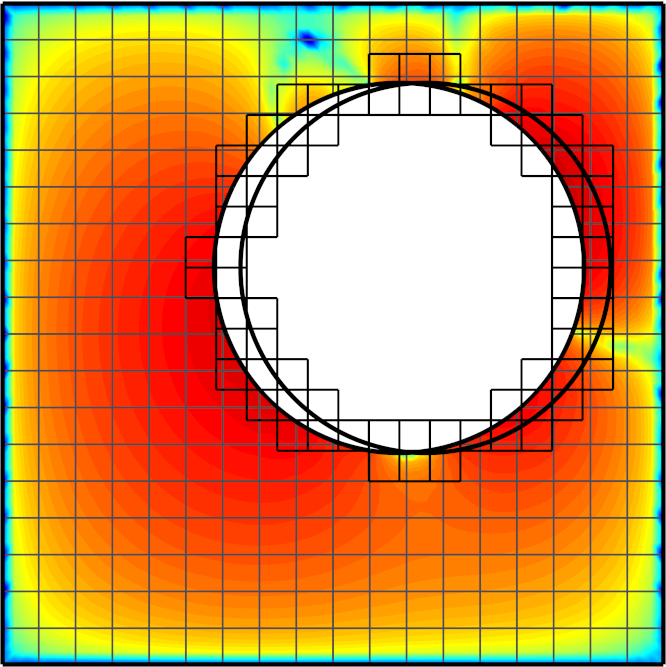}
\subcaption{$\log|u-u_h|$, $\delta h^{-1}=0.5$}
\label{fig:square:abserr:05}
\end{subfigure}
\begin{subfigure}[t]{0.3\linewidth}\centering
\includegraphics[width=0.9\linewidth]{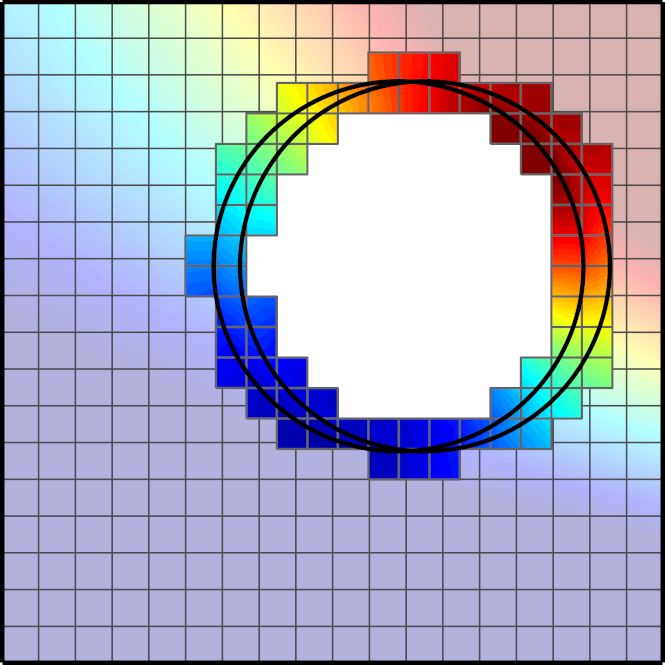}
\subcaption{$u_{h,0}$, $\delta h^{-1}=0.5$}
%\label{fig:square:blocksol:high}
\end{subfigure}
\\[0.5em]
\begin{subfigure}[t]{0.3\linewidth}\centering
\includegraphics[width=0.9\linewidth]{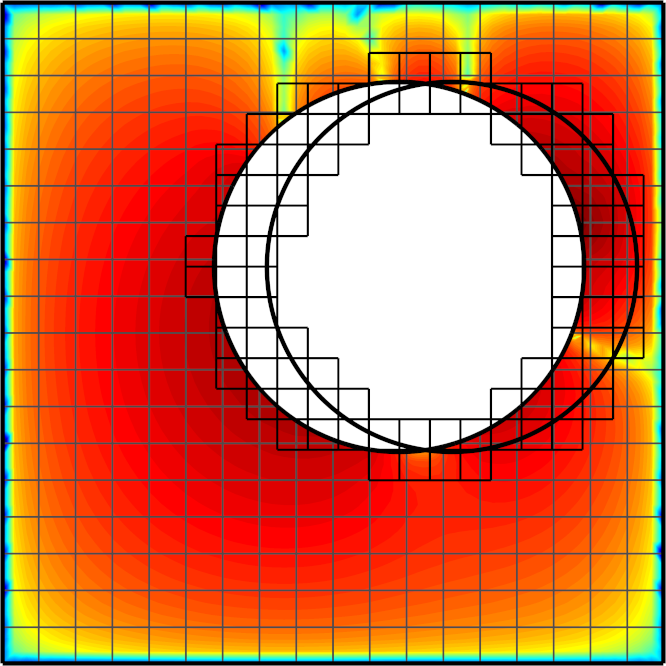}
\subcaption{$\log|u-u_h|$, $\delta h^{-1}=1$}
\label{fig:square:abserr:1}
\end{subfigure}
\begin{subfigure}[t]{0.3\linewidth}\centering
\includegraphics[width=0.9\linewidth]{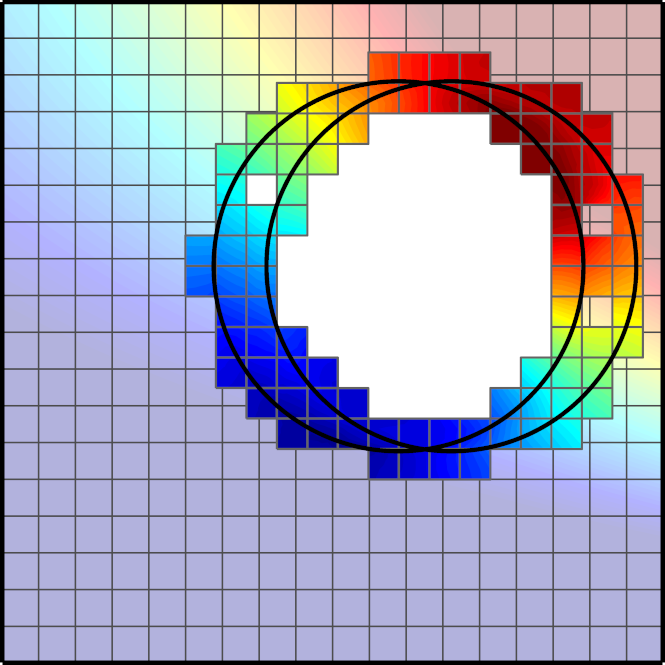}
\subcaption{$u_{h,0}$, $\delta h^{-1}=1$}
%\label{fig:square:blocksol:high}
\end{subfigure}
\caption{\emph{Effect of gap increase.} Sequence of two-patch domain with a gap, where the gap size $\delta$ is gradually increased. \emph{Left:} The absolute error for the numerical solution in the outer patch. \emph{Right:} The numerical solution of the hybrid variable on top of the (faded) numerical solution in the outer patch.}  
\label{fig:increasing-gap}
\end{figure}

\paragraph{Surface CAD Example.}
As a final example, we consider the surface CAD geometry of a tube intersection presented in Figure~\ref{fig:cad-surface}. This geometry was created using the surface CAD modeling software Rhino \cite{rhino} and exported in IGES format. In the CAD each of the three tubes is described as a parametric mapping from $[0,1]^2$ onto a tube surface along with trim curves in $[0,1]^2$ defining parts of the tube surface to remove, which in this case is given by the tube intersections.  Note that this surface CAD description does not include any connectivity information. To emphasize the gaps along the interfaces, we manually shifted the tube pieces for the final geometry.

In Figure~\ref{fig:cad-example} we present a numerical solution to a Dirichlet problem without load, where we impose different constant values on each of the four tube ends. Looking at the surface solution we note that it seems to flow nicely over the gaps. The hybrid variable stabilization seems to do its job as the hybrid variable solution does not appear to vary across the gap. Due to the exaggerated gap size a quite large mesh size is used for the hybrid variable, and, while seemingly not problematic in this example, we realize that the hybrid variable actually has some unwanted coupling between the various interfaces. This is a potential drawback of the simple implementation of the method where we define the hybrid variable for all interfaces using one continuous field. On the other hand, in practice this is not an issue for a problem with a more reasonable gap size and the simple and robust implementation are strengths of the method.

\begin{figure}
\centering	
\includegraphics[width=0.35\linewidth]{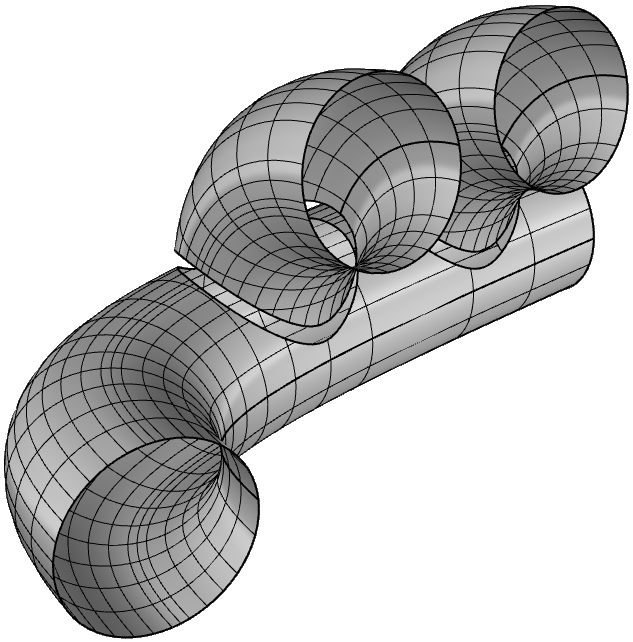}
\caption{\emph{CAD surface with gaps.}
The three tube pieces are patchwise described, where each piece is given by a parametric mapping from $[0,1]^2\subset\IR^2$ onto the tube surface. The interfaces where the tubes intersect are described using trim curves in each patch, defining parts of the surface to remove. Since the trim curves only give approximations to the true interfaces, the CAD surface includes small gaps over the interfaces.
In this example, we have exaggerated the gaps by translating the upper tubes vertically.}
\label{fig:cad-surface}
\end{figure}

\begin{figure}
\centering	
\begin{subfigure}[t]{0.45\linewidth}\centering
\includegraphics[width=0.86\linewidth]{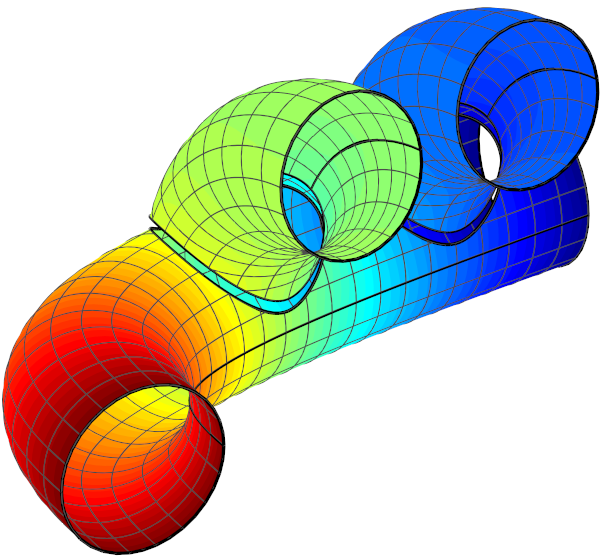}
\subcaption{Surface solution}
\label{fig:cad-uh}
\end{subfigure}
\begin{subfigure}[t]{0.45\linewidth}\centering
\includegraphics[width=0.9\linewidth]{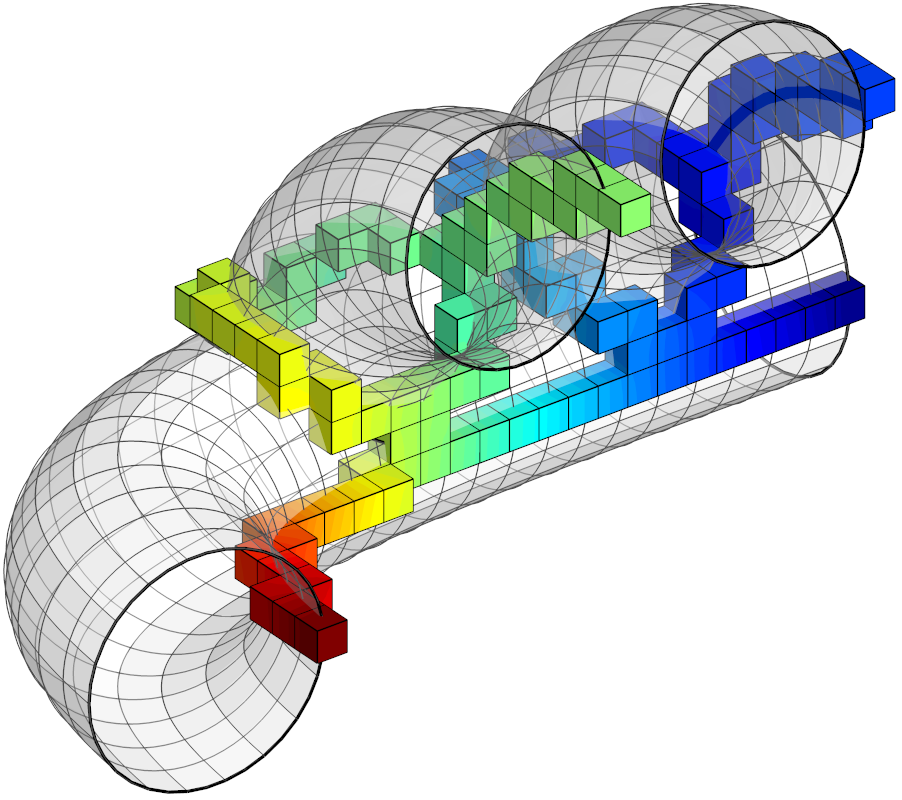}
\subcaption{Hybrid variable solution}
\label{fig:cad-uh0}
\end{subfigure}
\caption{\emph{Solution on CAD surface.}
In (a) we see the numerical solution to a Dirichlet problem with zero load where a constant value is imposed along each tube ending, with a different value for each ending. The solution flows nicely over the interfaces, both when coupling different patches over the gaps and when coupling a patch to itself. 
As seen in (b), both types of couplings are handled by the hybrid variable.
}
\label{fig:cad-example}
\end{figure}

\section{Summary}
In this contribution we have utilized weak enforcement techniques for interface problems, based on a hybridized Nitsche's formulation, to robustly couple solutions on surface CAD geometries with gaps/overlaps at the interfaces. Our approach has several benefits:
\begin{itemize}
\item \emph{Convenient and robust implementation.} 
The use of a hybridized Nitsche formulation for the coupling makes for a very convenient and robust implementation.
The convenience lies in that surface patches only directly couple to the hybrid variable, so assembly is naturally done patchwise, and that the hybrid variable is defined on a structured (potentially octree) grid in an embedding Euclidean space, which makes operations such as identifying in what element a point is located, easy and efficient.
In contrast to other multipatch methods based on Nitsche formulations, the hybrid formulation limits the need for computing inverses of the NURBS mappings in adjacent patches, which increases robustness.
Further, the use of CutFEM techniques in the patches makes for very flexible and convenient discretization choices, since the computational meshes are not required to conform to the trimmed reference domains.

\item \emph{Ease of application.} Since surface CAD models do not always include good connectivity information, i.e., the topological relationship of how the patch boundaries are coupled to each other, it significantly simplifies the application of the method that this information is not needed, but is rather implicit through the hybrid variable. Actually, the method is agnostic to both the number of surface patches joining at an interface, and whether the interface couples a patch to itself or to another patch.

\item \emph{Mathematical and numerical analysis.} Our preliminary mathematical analysis shows that we can devise an optimal order method using this technique, how the error is affected by the gap size, and what a suitable scaling of the hybrid variable stabilization is.
Our numerical results give further verification of the performance and insights into the behavior of the method.
\end{itemize}
A limitation in our current extraction of the hybrid variable mesh $\mcT_{h,0}$ is that we essentially assume the gap size $\delta$ to be smaller than the mesh size $h$, since there, from an accuracy perspective, is little motivation to use smaller $h$. However, it would be interesting to make the method robust also when $h \ll \delta$. This would require a technique for estimating the gap sizes along with an approach for padding $\mcT_{h,0}$ such that it always is simply connected across the gap.

\bigskip
\paragraph{Acknowledgement.} This research was supported in part by the Swedish Research
Council Grants Nos.\  2017-03911,  2021-04925,  and the Swedish
Research Programme Essence.

\bibliographystyle{habbrv}
\footnotesize{
\bibliography{ref}
}

%\vfill
\bigskip
\bigskip
\noindent
\footnotesize {\bf Authors' addresses:}

\smallskip
\noindent
Tobias Jonsson  \quad \hfill \addressumushort\\
{\tt tobias.jonsson@umu.se}

\smallskip
\noindent
Mats G. Larson,  \quad \hfill \addressumushort\\
{\tt mats.larson@umu.se}

\smallskip
\noindent
Karl Larsson, \quad \hfill \addressumushort\\
{\tt karl.larsson@umu.se}

\end{document}